\documentclass[12pt]{article}

\usepackage{amsmath,amsxtra,amssymb,latexsym,color,epsfig,amscd,amsthm,subfigure,fancybox}
\usepackage[mathscr]{eucal}
\usepackage{graphicx}
\usepackage{cases}
\usepackage{enumerate}
\newcommand{\wdt}{\widetilde}

\allowdisplaybreaks[2]

\newtheorem{thm}{Theorem}[section]

\newtheorem{lem}[thm]{Lemma}

\newtheorem{rem}[thm]{Remark}

\newtheorem{exm}[thm]{Example}

\newcommand{\eps}{\varepsilon}

\newcommand{\p}{\mathfrak{p}}

\newcommand{\M}{\mathcal{M}}
\newcommand{\m}{\mathfrak{m}}
\newcommand{\F}{\mathcal{F}}

\newcommand{\E}{\mathbb{E}}
\newcommand{\I}{\mathbf{I}}

\newcommand{\N}{\mathbb{Z_+}}
\newcommand{\PP}{\mathbb{P}}

\newcommand{\R}{\mathbb{R}}

\newcommand{\Lom}{\mathcal{L}}

\numberwithin{equation}{section}
\newcommand{\1}{\boldsymbol{1}}
\newcommand{\tr}{{\rm tr}}

\newcommand{\al}{\alpha}

\newcommand{\bed}{\begin{displaymath}}
\newcommand{\eed}{\end{displaymath}}
\newcommand{\bea}{\bed\begin{array}{rl}}
\newcommand{\eea}{\end{array}\eed}
\newcommand{\ad}{&\!\!\!\disp}
\newcommand{\aad}{&\disp}
\newcommand{\barray}{\begin{array}{ll}}
\newcommand{\earray}{\end{array}}
\newcommand{\diag}{{\rm diag}}
\def\disp{\displaystyle}
\newcommand{\rr}{{\Bbb R}}
\newcommand{\mz}{{m_0}}

\def\a.s{\text{\;a.s.\;}}

\newcommand{\nd}{\noindent}
\begin{document}

\title{Certain Properties Related to Well Posedness of Switching Diffusions}
\author{Dang Hai Nguyen,\thanks{Department of Mathematics, Wayne State University, Detroit, MI
48202,
dangnh.maths@gmail.com. This research  was supported in part
by the National Science Foundation under grant  DMS-1207667.}
\and George Yin,\thanks{Department of
Mathematics, Wayne State University, Detroit, MI 48202,
gyin@math.wayne.edu. This research  was supported in part
by the National Science Foundation under grant  DMS-1207667.}
\and Chao Zhu\thanks{Department of Mathematical
Sciences, University of Wisconsin-Milwaukee, Milwaukee,
WI 53201, zhu@uwm.edu.} }
\maketitle

\begin{abstract}
This work is devoted to switching diffusions  that
have two components (a continuous component and a discrete component).
Different from the so-called Markovian switching diffusions, in the setup,
the discrete component (the switching) depends on the continuous component (the diffusion process).
The  objective of this paper is to provide a number of properties related to the well posedness.
First, the differentiability with respect to initial data of the continuous component is established.
Then, further properties including uniform continuity with respect to
initial data, and smoothness of certain functionals are obtained. Moreover,
Feller property is obtained under only local Lipschitz continuity.
Finally, an example of Lotka-Voterra model under regime switching is provided
as an illustration.

\medskip
\nd {\bf Keywords.} Switching diffusion, continuous-state dependent switching, smoothness, Feller property.

\medskip
\nd{\bf Mathematics Subject Classification.} 60J05, 60J60.

\medskip
\nd {\bf Short Title.}  Continuity and Smoothness of Switching Diffusions
\end{abstract}

\newpage
\section{Introduction}
In the past two decades, a considerable research effort has been devoted to the study of switching diffusion processes that are also called hybrid switching diffusions.
Much of the interest stems from pressing need of treating complex systems involving both continuous dynamics representable by using
solutions of stochastic differential equations, and discrete events that cannot be written as solutions of the usual differential equations.
Such hybrid systems are prevalent in a wide range of applications including
ecological and biological modeling \cite{ZY09},
control systems and filtering \cite{Zhang98},
economics and finance \cite{Liu14,Zhang01},
networked systems \cite{KGM14}, among others.
These switching diffusions can be represented by a two-component process
$(X(t),\al(t))$, where $X(t)$ is a continuous component
(also called  ``continuous state'') taking values in $\rr^r$ and $\al(t)$ is a
discrete component
(also called
``discrete state'') taking values in a finite set $\M=\{1,\dots, m_0\}$.
The interactions of the continuous and discrete components make the models more versatile and suitable for a wide range of applications. On the other hand, these interactions make the analysis of such processes much more difficult. It is interesting to note that such processes are similar to the usual diffusions, but they could behave much differently from the usual diffusion processes qualitatively.
For example, suppose that we have two linear diffusions together with a continuous-time Markov chain. The Markov chain serves as a modulating force making the process switch back and forth between these two diffusions.
Depending on the switching frequency,
even if each of the diffusion is stable, the switched system can be unstable or vice  versa; see for example, \cite[pp. 229-233]{YinZ13}.
Likewise, we may have both of the individual diffusions
being recurrent, but the switching diffusion is not.
In a way, the switching diffusion processes display many peculiar properties.

Because of their importance,  hybrid systems in general and switching diffusions in particular have drawn resurgent attentions. For some of the recent progress, we mention the work \cite{BSY16,HMS,Shao-2,Xi08,Xi09,YM,ZY} and references therein. Systematic treatments and  comprehensive study of
these stochastic processes
  can be found
  in  \cite{MaoY} and \cite{YinZ}.
  Both of these references consider systems that have pure jump random switching in addition to the noise processes driven by Brownian motions;
the first reference mentioned above concentrates on switching diffusions in which the switching mechanism is given
 by a continuous-time Markov chain independent of the Brownian motion, whereas the second reference focuses on the switching processes depending  on the current state of the diffusions. In what follows, to distinguish these two types of switching diffusions,  the first type process
 is  referred to as
 Markovian switching diffusions (or Markov chain modulated switching diffusions), and
the second  type process is called  continuous-state-dependent switching diffusions.

This paper aims to study continuous and smooth dependence on the initial data of solutions to stochastic differential equations corresponding to  continuous-state-dependent switching diffusions.  These properties are all related to the well-posedness  in certain sense. In the book of Applebaum \cite{APPLEBAUM}, similar properties for stochastic differential equations, are also referred to as  flow properties.
These properties vividly highlight  the distinctions  between  Markovian switching diffusions and  continuous-state-dependent switching diffusions.
For example, it is well-known
that  a diffusion process is smooth in the $L^2$ sense with respect to its initial data under suitable conditions; see   \cite[VII, Section 4]{GS}.
Such a smoothness property readily carries over to the case of Markovian switching diffusions because  the Markov chain is independent of the Brownian motion and the initial state of the continuous component does not influence the dynamics of the Markov chain.
In this paper, we show that
 this phenomenon becomes markedly different
 in the case of
   continuous-state-dependent switching diffusions, which is in sharp
contrast  with that of   Markovian switching diffusions and diffusions.
In lieu of $L^2$ convergence, we demonstrate that
 the smoothness with respect to initial data is in the sense of $L^p$ for any $0<p$ strictly less than $\lambda$, with $\lambda\le 1$ being the H\"older exponent of the generator for the switching process; see Theorem \ref{thm-1} for the precise statement.
  Moreover, we provide a counterexample   to demonstrate that the estimates above is in fact sharp; we cannot expect $L^1$ convergence, neither can we get $L^2$ convergence in the current case.
 Next, we examine uniform estimates of two solutions if their initial data are close. Furthermore, we establish smoothness of certain functional of switching diffusions. As an application of the uniform estimates, we revisit the issue of Feller property.  Although Feller properties for continuous-state-dependent switching diffusions have been established in  the literature, see, for example, \cite{Xi08,Xi09} as well as Chapter 2 in \cite{YinZ}, this work aims to relax the commonly used global Lipschitz condition and to provide a simple proof.

The rest of the paper is arranged as follows. Section \ref{sec:for} presents the switching diffusion setup. Section \ref{sec:diff} focuses on the differentiability of the switching diffusions with respect to the initial data
of the continuous state; also presented in this section is a counter-example of the differentiability under $L^1$ convergence.
Section \ref{sec:fur} studies further properties, in which Section
\ref{sec:con} furthers our investigation on uniform estimates on a finite time interval of solutions with different initial data
 and Section \ref{sec:fun} investigates smoothness of a functional of the switching diffusions. Using results obtained in Sections \ref{sec:diff}, in Section \ref{sec:fel}, we first provide an alternative proof of the Feller property under global Lipschitz condition. Then we obtain the Feller property by using only local Lipschitz continuity. Finally, we close the paper with an example in Section \ref{sec:lot}, a competitive regime-switching Lotka-Volterra model, which illustrates our results.

\section{Switching Diffusions}\label{sec:for}
Let $(\Omega,\F,\F_t,\PP)$ be a filtered probability space, where $\Omega$ is the sample space, $\F$ is the $\sigma$-algebra of subsets, $\{\F_t\}$
is a  filtration  (that is, $\{\F_t\}$ is a family of $\sigma$-algebras  satisfying $\F_s \subset \F_t$ for $s\le t$), and $\PP$ is a probability measure.
We assume that $\F_t$ is complete in that it contains all null sets, and that $\F_t$ satisfies the usual condition in that $\F_0$ is complete and $\{\F_t\}$ is right continuous.
Let $\M=\{1,\dots,m_0\}$ be a finite set, and
suppose that $b(\cdot,\cdot): \R^r\times\M\to\R^r$ and $\sigma(\cdot,\cdot): \R^r\times\M\to\R^{r\times d}$.
In this paper, we
consider  a switching diffusion process, a two-component Markov process $(X(t),\alpha(t))$ whose generator is given by
\begin{equation}\label{oper-sd}\barray \disp
{\cal L}f(x,i )\ad =
\nabla f'(x,i ) b(x,i )
 + \tr (\nabla^2 f(x,i ) A(x,i )) + Q(x)
f(x,\cdot)(i )\\
\ad =
\sum^r_{k=1}b_k(x,i )
\frac{\partial f(x,i )}{\partial x_k}
 +\frac{1}{2}
 \sum^r_{k,l=1}a_{kl}(x,i )
\frac{\partial^2 f(x,i )}{\partial x_k \partial x_l}\\
\aad \ + Q(x)
f(x,\cdot)(i ), \ \hbox{ for any } \ (x, i ) \in  \rr^r \times \M,\earray
\end{equation}
for a $C^2$-function $f(\cdot, i )$ (whose derivatives with respect to $x$ up to the second order are continuous)
for each $i  \in \M$,
where
$\nabla f(x,i )$ and $\nabla^2 f(x,i )$
denote the gradient
and Hessian of $f(x,i )$ with
respect to $x$, respectively, and
\bea\ad Q(x) f(x,\cdot) (i ) =\sum^\mz_{j=1} q_{i  j}(x) f(x,j),
\ \hbox{ and } \\
\ad A(x,i)=(a_{kl}(x,i
))=\sigma(x,i)\sigma'(x,i)\in \rr^{r\times r}.\eea
The dynamics and transition rules of $(X,\al)$ may also be presented as follows.
Suppose that
$w(t)$ is an
$\R^d$-valued Brownian motion,
 $\alpha(t)$ is a pure jump process taking value in  $\M$,  and $X(t)$ satisfies
\begin{equation}\label{sde}
dX(t)=b(X(t), \alpha(t))dt+\sigma(X(t),\alpha(t))dw(t),\end{equation}
such that
 the jump intensity of $\alpha(t)$ depends on the current state of $X(t)$ in that
 the generator of $\alpha(t)$ is given by $Q(x)=(q_{ij}(x))$
 with $q_{ij}(\cdot):\R^r\to\R$ for $i,j\in\M$,  $q_{ij}(x)\ge 0$ for $i\not =j$, and
$\sum_j q_{ij}(x)=0$ for each $i\in \M$, satisfying
\begin{equation}\label{tran}
\begin{array}{ll}
&\disp \PP\{\alpha(t+\Delta)=j|\alpha(t)=i, X(s),\alpha(s), s\leq t\}=q_{ij}(X(t))\Delta+o(\Delta) \text{ if } i\ne j
.\end{array}\end{equation}
Alternatively,
the evolution of the discrete component $\al(\cdot)$
can be represented by a stochastic integral with respect to a Poisson random
measure (see, e.g., \cite{Skorohod-89}).
For
$x\in \rr^r$ and $i
\in \M$, let $\Delta_{ij}(x), j\in\M\setminus\{i\}$ be
disjoint intervals of
the real line, each having length $q_{ij}(x)$. 
Suppose $|q_{ij}(x)|<M$ for any $i\ne j$,
then we can choose $$\Delta_{ij}(x)=\Big[((i-1)m_0+j)M, ((i-1)m_0+j)M+q_{ij}(x)\Big).$$
Define a function
$h: \rr^r \times \M \times \rr \mapsto \rr$ by
\begin{equation}\label{h-def} h(x,i,z)=\sum^\mz_{j=1} (j-i)
I_{\{z\in\Delta_{ij}(x)\}}.
\end{equation}
That is, with the partition $\{\Delta_{ij}(x): i,j\in\M\}$
used and for each $i\in
\M$,
if $z\in\Delta_{ij}(x)$, then $h(x,i,z)=j-i$; otherwise $h(x,i,z)=0$.
Then the dynamics of
the discrete component $\al(\cdot)$ can be represented by
\begin{equation}\label{eq:ju} d\al(t) = \int_\rr h(X(t),\al(t-),z) {\mathfrak p}(dt,dz),\end{equation}
where ${\mathfrak p}(dt,dz)$ is a Poisson random measure with intensity $dt
\times \m(dz)$, and $\m$ is the Lebesgue measure on $\rr$. The Poisson random
measure ${\mathfrak p}(\cdot, \cdot)$ is independent of the Brownian motion
$w(\cdot)$.
Thus, the switching process  $(X,\al)$  can be presented by the system of stochastic differential equations  given by \eqref{sde} and \eqref{eq:ju}. Let the initial condition of the switching diffusion be $(X(0),\al(0)) =(x, \al) \in \R^{r}\times \M$.
To ensure the existence and uniqueness of solutions to \eqref{sde} and \eqref{tran},
we have the following theorem, which is proved in \cite[Theorem 2.7]{YinZ}.

\begin{thm}\label{exist-unique}
Assume that $b(x,i)$ and $\sigma(x,i)$ are locally Lipschitz  in $x$ for each $i\in\M$, and  that $Q(x)=(q_{ij}(x))$ is bounded and continuous.
Assume further that there exists a function
$V(\cdot,\cdot):\R^r\times\M\mapsto\R_+$ that is twice continuously differentiable with respect
to $x\in\R^r$ for each $i\in\M$ and a constant $K>0$ such that
$$
\begin{cases}
{\cal L}V(x,i)\leq K V(x,i)\,\text{ for all }\, (x,i)\in\R^r\times\M\\
V_R:=\inf\{V(x,i): |x|\geq R, i\in \M\} \to \infty\,\text{ as }\, R\to\infty
\end{cases}
$$
Then  the system given by \eqref{sde} and \eqref{tran} has
a unique  strong solution for each initial condition.
\end{thm}

Note that for Markovian switching diffusions, the switching process is a
 homogeneous continuous-time Markov chain with a constant matrix as its generator, and
 the Markov chain and the Brownian motion are independent.
 In lieu of such a structure, we are dealing with a much more complex system. Starting from  the next section, we focus on deriving certain smoothness and continuity properties of the underlying processes.

\section{Differentiability with Respect to Initial Data}\label{sec:diff}
This section is devoted to differentiability with respect to initial data of the switching processes, or equivalently, the smoothness of the solutions of the switched stochastic differential equations with respect to the initial data of the continuous component.
This section is divided into two parts. The first part derives the main result. The second part presents an counter example indicating our result is sharp.

\subsection{Differentiability}
In \cite[Theorem 4.2]{YZ}, we stated that $X(t)$ is twice continuously differentiable in mean square under certain conditions.
There is an error in the proof of \cite[Lemma 4.3]{YZ}. In particular, in the last line of equation (4.13) on page 2421 of the aforementioned paper,
a diagonal entry term of the form
$\wdt q_{(j,l)(j,l)}(x,\wdt x)$ in the generator of the coupling was inadvertently left out
in our proof,
resulting an error in the proof of the said theorem. We correct this error here and
 demonstrate that the switching diffusions do possess the smoothness properties. However the results about
the differentiability in mean square is in general  {\em not achievable} if the generator $Q(x)=(q_{ij}(x))$ of $\alpha(t)$  depends on the current state of $X(t)$; see the simple yet illuminating example in Section \ref{sec:c-exm}.
The differentiability in mean square needs to be replaced by differentiability in $p$th moment for an appropriate $p$.
It turns out that the proof is interesting in its own right. In a way, it really displays certain aspects of the salient features of
the continuous state-dependent switching processes.

 If we consider a Markovian switching diffusion process, the switching times and hence the switching process can be generated beforehand because $\alpha(t)$ is independent of the Brownian motion. As can be seen in the proof of this section, the main difficulty we face is that the switching times depend on the continuous component $X(t)$.
 To illustrate, consider a one-dimensional $X(t)$ as an example.
 Denote the solution of \eqref{sde} with two different initial data for the continuous component $(X(0),\alpha(0))=(x,\alpha)$ and $(X(0),\alpha(0))=(\wdt x,\alpha)=(x+\Delta, \alpha)$
 by $X^{x,\alpha}(t)$ and $X^{\wdt x,\alpha}(t)$, respectively. The differentiability is concerned with the limit of the difference quotient $\frac{X^{\wdt x,\alpha}(t) -X^{x,\alpha}(t)}{\wdt x -x }$.
 In the Markovian   switching diffusions, the difference $X^{\wdt x,\alpha}(t) -X^{x,\alpha}(t)$ can be calculated much the same way as in the diffusion case in \cite[VIII, Section 4, pp. 403-412]{GS}. That is,
 we can simply subtract one from the other. The $\alpha(t)$ does not really come into the picture because even $x\not = \wdt x$, the sample paths of $\alpha(t)$ are the same.
 For our continuous-state dependent switching,
  care must be taken. The analysis is more delicate in places because  $\alpha^{x,\alpha}(t)$ and $\alpha^{\wdt x, \alpha}(t)$ can take different values infinitely often.
  In the analysis to follow,
  one of the main insights is the use of
   the first time when  the switching processes $\al^{x,\alpha}(t)$ and $\alpha^{\wdt x,\al}(t)$ are different.

To proceed, we first setup the notation.
For a multi-index or a  vector $\beta=(\beta_1,\cdots,\beta_r)$ with nonnegative integer entries, put $|\beta|=\sum_{i=1}^r\beta_i$ and define
$$D^\beta_x=\dfrac{\partial^{|\beta|}}{\partial x_1^{\beta_1}\dots\partial x_r^{\beta_r}}.$$
Next we define the $L^p$ differentiability of a smooth random function
$\Phi(x_1,\dots,x_r,t)$. Its partial derivative in probability
with respect to $x_i$ is defined as a random variable $\Psi(x_1,\dots,x_r,t)$ such that
$$\dfrac1h\Big(\Phi(x_1,\dots,x_i+h, \dots, x_r,t)-\Phi(x_1,\dots,x_i, \dots, x_r,t)\Big)\to \Psi(x_1,\dots,x_r,t)\text{ in probability} $$ as $h\to0$.
If for some $p>0$,
$$\E\left|\dfrac1h\Big(\Phi(x_1,\dots,x_i+h, \dots, x_r,t)-\Phi(x_1,\dots,x_i, \dots, x_r,t)\Big)-\Psi(x_1,\dots,x_r,t)\right|^p\to0\text{ as }h\to0$$
we say that $\Phi(x_1,\dots,x_r,t)$ has partial derivative
with respect to $x_i$ in $L^p$.
We proceed to obtain the smoothness of the switching diffusion with respect to the initial data in the $L^p$ sense.

\begin{thm}\label{thm-1}
Assume that $b(x,i)$ and $\sigma(x,i)$ are Lipschitz  in $x$ for each $i\in\M$.
Let $(X^{x,\alpha}(t),\alpha^{x,\alpha}(t))$ be the solution to the system given by \eqref{sde} and \eqref{tran}
with initial data $(X(0),\alpha(0))=(x,\alpha)$. Assume that for each $i\in\M$, $b(\cdot,i)$ and $\sigma(\cdot,i)$ have continuous partial derivatives with respect to the
variable $x$ up to the second order and that
\begin{equation}\label{e.bd}
|D^\beta_xb(x,i)|+|D^\beta_x\sigma(x,i)|\leq K_0(1+|x|^{\gamma_0}),
\end{equation}
where $K_0$ and $\gamma_0$ are positive constants and $\beta$ is a multi-index with $|\beta|\leq 2$.
\begin{itemize}
\item[{\rm (a)}] Suppose that
$Q(x)=(q_{ij}(x))$ is bounded and continuous.
 Then $X^{x,\alpha}(t)$ is
twice continuously differentiable with respect to $x$ in probability.
\item[{\rm (b)}] Replace the conditions for $Q(x)$
in the assumptions of part {\rm(a)}  by $Q(x)$ is bounded and
$q_{kj}(x)$ is locally H\"older continuous with H\"older exponent $\lambda$ for some $\lambda\in(0,1]$ in that
 there are $K_1>0$ and $\gamma_1>0$  such that
\begin{equation}\label{holder}
\sum_{k,j\in\M}|q_{kj}(x)-q_{kj}(y)|\leq K_1(1+|x|^{\gamma_1}+|y|^{\gamma_1})|x-y|^\lambda,\,\forall x,y\in\R^r.
\end{equation}
Then $X^{x,\alpha}(t)$ is
twice continuously differentiable in $L^p$ with respect to $x$ for any $0<p<\lambda$.
\end{itemize}
\end{thm}

\begin{proof}
For ease of presentation and without loss of generality, we prove the result when $X(t)$ is one dimensional.
Fix $(x,\alpha)\in\R\times\M$ and $T>0$.
Let $(X(t),\alpha(t))$ be the switching-diffusion process satisfying \eqref{sde} and \eqref{tran} with
initial condition $(x,\alpha)$ respectively. Likewise, let $(\wdt  X(t),\wdt  \alpha(t))$
be the solution process with initial condition $(\wdt  x,\alpha)$, where $\wdt  x=x+\Delta$ for $0<|\Delta| \ll 1$.
Consider the joint process $(X(t),\wdt  X(t), \alpha(t), \wdt \alpha(t))$.
By the basic coupling method (see e.g., \cite[p. 11]{Chen}),
we can consider them as the solutions to
\begin{equation}\label{e2.3}
\begin{cases}
dX(t)=b(X(t), \alpha(t))dt+\sigma(X(t),\alpha(t))dw(t) \\
d\wdt  X(t)=b(\wdt X(t), \wdt \alpha(t))dt+\sigma(\wdt  X(t),\wdt \alpha(t))dw(t)
\end{cases}
\end{equation}
with initial condition $(x,x+\Delta,\alpha,\alpha)$, where $(\alpha(t),\wdt \alpha(t))$ has the generator $\wdt Q(X(t), \wdt  X(t))$ which is defined by
\begin{equation}\label{eq:coupled-Q}
\begin{split}
\wdt  Q(x,\wdt  x)\wdt  f(k,l)=&\sum_{(j,i\in\M\times\M)}\wdt q_{(k,l)(j,i)}(x,\wdt  x)\left(\wdt  f(j,i)-\wdt  f(k,l)\right)\\
=&\sum_{j\in\M}[q_{kj}(x)-q_{lj}(\wdt  x)]^+(\wdt  f(j,l)-\wdt  f(k,l))\\
&+\sum_{j\in\M}[q_{lj}(\wdt x)-q_{kj}(  x)]^+(\wdt  f(k,j)-\wdt  f(k,l))\\
&+\sum_{j\in\M}[q_{kj}(x)\wedge q_{lj}(\wdt  x)](\wdt  f(j,j)-\wdt  f(k,l)).
\end{split}
\end{equation}

Define
\begin{equation}\label{eq:tau-d-def}\tau^\Delta=\inf\{t\geq0: \alpha(t)\ne\wdt \alpha(t)\}.\end{equation}
Since $\alpha(t)$ and $\wdt \alpha(t)$ are the same up to $\tau^\Delta$,
and $b(\cdot,i)$ and $\sigma(\cdot,i)$ are Lipschitz continuous for each $i\in \M$,
by standard arguments (see e.g. \cite[Lemma 3.3]{MaoY}), we  obtain that
\begin{equation}\label{e3}
\E\sup\limits_{0\leq t\leq T\wedge\tau^\Delta}\{|X(t)-\wdt  X(t)|^2\}\leq K|x-\wdt x|^2.
\end{equation}
Recall from \cite[Lemma 3.2]{YZ} that for any $m>0$ and $0< R<\infty$, there is a $C_m(R,T)>0$ such that
for any pair $(x_0,\alpha_0)$ with $|x_0|\le R$ and $\alpha_0 \in \M$,
\begin{equation}\label{egrowth}
\E\sup_{t\in[0,T]}\{|X^{x_0,\alpha_0}(t)|^m\}\leq C_m(R, T) .
\end{equation}
First, we show that $$\PP\{\tau^\Delta\leq T\}\to0\text{ as } \Delta\to0.$$
Let $\wdt f(k,l)= \1_{\{k\not =l\}}$, where $\1_A$ is the indicator of the set $A$.
By the definition of the function $\wdt f$,  we have
\begin{equation}\label{e.rho}
\begin{aligned}
\wdt  Q(x,\wdt  x)\wdt  f(k,k)=&\sum_{j\in\M, j\ne k}[q_{kj}(x)-q_{kj}(\wdt  x)]^++\sum_{j\in\M, j\ne k}[q_{kj}(\wdt x)-q_{kj}(x)]^+\\
=&\sum_{j\in\M, j\ne k}|q_{kj}(x)-q_{kj}(\wdt  x)|=:\rho(x,\wdt  x, k).
\end{aligned}
\end{equation}
Applying the generalized It\^o formula to \eqref{e2.3} and noting that $\alpha(t)=\wdt \alpha(t), t<\tau^\Delta$,  we obtain that
\begin{equation}\label{e6}
\begin{split}
\PP\{\tau^\Delta\leq T\}&= \E \wdt f\left(\alpha(T\wedge\tau^\Delta),\wdt \alpha(T\wedge\tau^\Delta)\right)\\
&= \E\int_0^{T\wedge\tau^\Delta}\wdt Q(X(t),\wdt  X(t)) \wdt  f(\alpha(t),\wdt \alpha(t))dt
=\E \int_0^{T\wedge\tau^\Delta}\rho(X(t),\wdt  X(t), \alpha(t))dt.
\end{split}
\end{equation}
In view of \eqref{e3},
\begin{equation}\label{e4}
\sup\limits_{t\in[0,T\wedge\tau^\Delta]}|X(t)-\wdt  X(t)|\to 0\text{ in probability as } \Delta\to0.
\end{equation}
By the dominated convergence theorem,
it follows from \eqref{e4} and boundedness of  $\rho(\cdot,\cdot,\cdot)$
that
\begin{equation}\label{e7}
\begin{split}
\lim\limits_{\Delta\to0}\PP\{\tau^\Delta\leq T\}=\lim\limits_{\Delta\to0}\E \int_0^{T\wedge\tau^\Delta}\rho(X(t),\wdt  X(t), \alpha(t))dt=0
\end{split}
\end{equation}
Moreover, if \eqref{holder} is satisfied, then  by H\"older's inequality, estimates \eqref{holder}, and \eqref{egrowth}, we have
\begin{equation}\label{e8}
\begin{split}
\PP\{\tau^\Delta\leq T\}&=\E \int_0^{T\wedge\tau^\Delta}\rho(X(t),\wdt  X(t), \alpha(t))dt\\
&\leq \E \int_0^{T\wedge\tau^\Delta} K_1(1+|X(t)|^{\gamma_1}+|\wdt  X(t)|^{\gamma_1})|X(t)-\wdt  X(t)|^\lambda dt\\
&\leq K_2T\E \sup\limits_{t\leq T\wedge\tau^\Delta}(1+|X(t)|^{\gamma_1}+|\wdt  X(t)|^{\gamma_1})|X(t)-\wdt  X(t)|^\lambda \,\text{( for some } K_2>0)\\
&\leq K_2T\Big(\E \sup\limits_{t\leq T\wedge\tau^\Delta}(1+|X(t)|^{\gamma_1}+|\wdt  X(t)|^{\gamma_1})^{\frac{2}{2-\lambda}}\Big)^{\frac{2-\lambda}{2}}
\Big(\E \sup\limits_{t\leq T\wedge\tau^\Delta}|X(t)-\wdt  X(t)|^2\Big)^{\frac{\lambda}{2}}\\
&\leq \wdt  K_\lambda|\Delta|^\lambda\ \text{ (for some }\wdt  K_\lambda>0).
\end{split}
\end{equation}  As in \cite{YZ}, put $Z^{\Delta}(t) : = \frac{\wdt X(t)- X(t)}{\Delta}$ for $t \ge0$. Then we have
\begin{align*}
Z^{\Delta}(t\wedge \tau^{\Delta})  & =1 + \dfrac1\Delta \int_0^{t\wedge\tau^\Delta}\Big[b(\wdt  X(s),\alpha(s))-b(X(s),\alpha(s))\Big]ds \\
    & \qquad +  \dfrac1\Delta\int_0^{t\wedge\tau^\Delta}\Big[\sigma(\wdt  X(s),\alpha(s))-\sigma(X(s),\alpha(s))\Big]dw(s). \\
\end{align*}
Now, we evaluate the drift:
\begin{equation}\label{e.zdelta}
\begin{split}
\dfrac1\Delta& \int_0^{t\wedge\tau^\Delta}\left[b(\wdt  X(s),\alpha(s))-b(X(s),\alpha(s))\right]ds\\
&=\dfrac1\Delta\int_0^{t}\1_{\{s\leq \tau^\Delta\}}\left[b(\wdt  X(s),\alpha(s))-b(X(s),\alpha(s))\right]ds\\
&= \dfrac1\Delta\int_0^{t}\1_{\{s\leq\tau^\Delta\}}\left(\int_0^1\dfrac{d}{dv}b(X(s)+v(\wdt  X(s)-X(s)),\alpha(s))dv\right)ds \\
&= \int_0^{t} Z^\Delta(s)\1_{\{s\leq\tau^\Delta\}} \int_0^1 b_{x}(X(s)+v(\wdt  X(s)-X(s)),\alpha(s))dv ds.\\
&= \int_0^{t} Z^\Delta(s\wedge\tau^\Delta)\1_{\{s\leq\tau^\Delta\}} \int_0^1 b_{x}(X(s)+v(\wdt  X(s)-X(s)),\alpha(s))dv ds.
\end{split}
\end{equation}
In view of \eqref{e.zdelta} and a similar evaluation for the diffusion part, we have that
$U_\Delta(t)=Z(t\wedge\tau^\Delta)$ satisfies the equation
$$U_\Delta(t)=1+\int_0^tA_\Delta(s) U_\Delta(s)ds+\int_0^tB_\Delta(s) U_\Delta(s)dw(s)$$
where
$$A_\Delta(s) :=\1_{\{s\leq\tau^\Delta\}} \int_0^1 b_{x}(X(s)+v(\wdt  X(s)-X(s)),\alpha(s))dv,$$
$$B_\Delta(s) :=\1_{\{s\leq\tau^\Delta\}} \int_0^1 \sigma_{x}(X(s)+v(\wdt  X(s)-X(s)),\alpha(s))dv.$$
By \eqref{e4} and \eqref{e7}, as $\Delta\to0$,
$$A_\Delta(s) \to b_x(X(s),\alpha(s)) \text{ and } B_\Delta(s) \to \sigma_x(X(s),\alpha(s)) \text{ in probability}.$$
In view of \cite[Theorem 5.2.2]{AF},
\begin{equation}\label{e9}
\E\biggl|\dfrac{\wdt  X(T\wedge\tau^\Delta)- X(T\wedge\tau^\Delta)}{\Delta}-\xi(T)\biggr|^2=\E\bigl|U_\Delta(t)-\xi(T)\bigr|^2\to 0\text{ as } \Delta\to0.
\end{equation}
where
$\xi(t)$ is the solution to
\begin{equation}
\label{eq-xi}
\xi(t)=1+\int_0^tb_x(X(s),\alpha(s))\xi(s)ds+\int_0^t\sigma_x(X(s),\alpha(s))\xi(s)dw(s).
\end{equation}
Note that
\begin{equation}\label{e10}\barray
\dfrac{\wdt  X(T)-X(T)}{\Delta}\ad =\dfrac{\wdt  X(T\wedge\tau^\Delta)- X(T\wedge\tau^\Delta)}{\Delta}\\
\aad \ +\1_{\{\tau^\Delta\leq T\}}\dfrac{\wdt  X(T)-\wdt  X(T\wedge\tau^\Delta)- X(T)+ X(T\wedge\tau^\Delta)}{\Delta}.
\earray\end{equation}
In light of \eqref{e7},
\begin{equation}\label{e11}
\1_{\{\tau^\Delta\leq T\}}\dfrac{\wdt  X(T)-\wdt  X(T\wedge\tau^\Delta)- X(T)+ X(T\wedge\tau^\Delta)}{\Delta}\to 0\text{ in probability as } \Delta\to0.
\end{equation}
It follows from \eqref{e9}, \eqref{e10}, and \eqref{e11} that
\begin{equation}\label{e12}
\dfrac{\wdt  X(T)- X(T)}{\Delta}\to \xi(T)\text{ in probability as } \Delta\to0.
\end{equation}
This proves part (a) of the theorem.

Next we prove part (b) of the theorem.
For $p<\lambda$, letting $\theta=\frac{\lambda-p}2$, we have
\begin{equation}
\begin{aligned}
\E\1_{\{\tau^\Delta\leq T\}}&\left|\wdt  X(T)-\wdt  X(T\wedge\tau^\Delta)- X(T)+ X(T\wedge\tau^\Delta)\right|^p\\
\leq& 2^p \E\1_{\{\tau^\Delta\leq T\}}\Big(\sup_{t\in[0,T]}\{|X(t)|+|\wdt  X(t)|\}\Big)^p\\
\leq& 2^p (\E\1_{\{\tau^\Delta\leq T\}})^{\frac{p+\theta}\lambda}\left(\E\Big(\sup_{t\in[0,T]}\{|X(t)|+|\wdt  X(t)|\}\Big)^{\frac{p\lambda}\theta}\right)^{\frac{\theta}{\lambda}}\\
\leq&\wdt  K_2 |\Delta|^{p+\theta}\text{ (for some } \wdt  K_2\text{ using \eqref{egrowth} and \eqref{e8})},
\end{aligned}\label{e13}
\end{equation}
which implies that
\begin{equation}\label{e14}
\E\1_{\{\tau^\Delta\leq T\}}\Big|\frac{\wdt  X(T)-\wdt  X(T\wedge\tau^\Delta)- X(T)+ X(T\wedge\tau^\Delta)}\Delta\Big|^p\to 0\text{ as }\Delta\to 0.
\end{equation}
As a result of \eqref{e9}, \eqref{e10} and \eqref{e14}, we obtain that
\begin{equation}\label{e15}
\E\Big|\dfrac{\wdt  X(T)- X(T)}{\Delta}\to \xi(T)\Big|^p\to0\text{  as } \Delta\to0.
\end{equation}
Thus, $X(t)$ is differentiable in $L^p$ for $p<\lambda$ if \eqref{holder} is satisfied.
Likewise, it can be shown that $X(t)$ is twice differentiable in probability and in $L^p$ for $p< \lambda$ under the condition \eqref{holder}.
\end{proof}

\begin{rem}{\rm
In the proof, we use the  boundedness condition on $Q(x)$
to derive \eqref{e7} from \eqref{e4}  using the dominated convergence theorem.
If we assume the polynomial growth
$|q_{ij}(x)|\leq K_0(1+|x|^{\gamma_0})$,
we will have
$\rho(x,\tilde x,k)\leq \widetilde K_0(1+|x|+|\tilde x|)^{\gamma_0}$
for some $\wdt K_0>0$.
Thus, we have $$\E \left(\int_0^{T\wedge\tau^\Delta}\rho(X(t),\wdt  X(t), \alpha(t))dt\right)^2
\leq T^2\wdt K_0^2\E \sup_{t\in[0,T]}(1+|X(t)|+|\wdt  X(t)|)^{2\gamma_0},
$$
which results in the uniform integrability of $\left\{\int_0^{T\wedge\tau^\Delta}\rho(X(t),\wdt  X(t), \alpha(t))dt: \Delta\in(0,1]\right\}.$
As a result, we derive from \eqref{e4} and the Vitali convergence theorem
that \eqref{e7} still holds
if we replace the boundedness of $q_{ij}(x)$ by the condition
$|q_{ij}(x)|\leq K_0(1+|x|^{\gamma_0})$.
All the
remaining arguments in the proof of Theorem \ref{thm-1} still carry over.

}
\end{rem}

\subsection{A Counterexample for Smoothness under $L^1$ Convergence}\label{sec:c-exm}
The smoothness with respect to the initial data was obtained in the last section.
The convergence is in the sense of $L^p$ convergence. One immediate question is: Can we do better?
Is it possible to get, for example, $L^1$ convergence?
In general, this question has a negative answer.
The intuitive arguments are as follows.
$X(t)$ and $\wdt X(t)$ evolves
similarly up to $\tau^\Delta$, the moment $\alpha(t)$ and $\wdt\alpha(t)$
switch apart.
Thus, if $T\leq\tau^\Delta$,
it is easy to estimate
$
\frac{X(T)-\wdt X(T)}{\Delta}
$ using the (local) Lipschitz condition of the coefficients.
However, if $T> \tau^\Delta$,
we cannot expect that $
\frac{X(T)-\wdt X(T)}{\Delta}
$ is bounded
since $X(t)$ and $\wdt X(t)$ follow different equations after $\tau^\Delta$.
Using this observation, in the proof of Theorem \ref{thm-1},
we decompose $\frac{X(T)-\wdt X(T)}{\Delta}$
into two parts in \eqref{e10}.
The first term on the right-hand side of \eqref{e10}
converges to $\xi(T)$ in $L^2$
while the second term converges  to $0$ in probability.
Thus,
if $\frac{X(T)-\wdt X(T)}{\Delta}$ were to converge in $L^1$ as $\Delta\to0$,
we would expect that
$$\1_{\{\tau^\Delta\leq T\}}\dfrac{\wdt  X(T)-\wdt  X(T\wedge\tau^\Delta)- X(T)+ X(T\wedge\tau^\Delta)}{\Delta}
\to 0\,\text{ in } L^1.
$$
Since $X(t)$ and $\wdt X(t)$ evolve completely differently after $\tau^\Delta$,
we cannot expect that $X(T)-\wdt X(T)\to 0$ as $\Delta\to 0$ in the event $\{T>\tau^\Delta\}$.
Thus,
\begin{equation}\label{e-i1}
\dfrac{\wdt  X(T)-\wdt  X(T\wedge\tau^\Delta)- X(T)+ X(T\wedge\tau^\Delta)}{\Delta}=O(\Delta^{-1}) \text{ as } \Delta\to0\,\text{ if }\,  T>\tau^\Delta.
\end{equation}
On the other hand,
it follows from \eqref{e.rho}, \eqref{e6}, and
the H\"older continuity \eqref{holder}
that
\begin{equation}\label{e-i2}
\PP\{T\geq\tau^\Delta\}=O(\Delta^{\lambda}) \text{ as } \Delta\to0.	
\end{equation}
As a result of \eqref{e-i1} and \eqref{e-i2},
we can see that $\1_{\{\tau^\Delta\leq T\}}\frac{\wdt  X(T)-\wdt  X(T\wedge\tau^\Delta)- X(T)+ X(T\wedge\tau^\Delta)}{\Delta}$ does not in general
converge to $0$ in $L^p$ for $p\geq\lambda$.
Because the H\"older exponent $\lambda$ cannot exceed $1$ except for the case when
the $Q$-matrix is constant,
 we cannot, in general, obtain the $L^1$ convergence.

In this section, we provide an example
showing that the process $X^{x,\alpha}(t)$ is not differentiable in $L^2$ or in $L^1$ although
the coefficients $b(\cdot,\cdot),\sigma(\cdot,\cdot)$ and   $Q(x)$ are smooth and bounded.
In order to simplify the calculations,
we consider an example in which
the $Q$-matrix is reducible and there is no diffusion part.
It is possible to construct a counter-example
with irreducible $Q$-matrix and nondegenerate diffusion.
However, it will involve more cumbersome and tedious calculations.
It appears to be more instructive to construct an example with structure as simple as possible to highlight the distinctions of $x$-dependent switching and Markovian switching.
Thus, we omit such an example here.

Let $\M=\{1,2\}$ and consider the equation
\begin{equation}\label{eq-sde-ex}
dX(t)=b(\alpha(t))dt
\end{equation}
where
$b(1)=0$ and $b(2)=1.$
Suppose that the switching process $\alpha(t)$ has generator
$Q=\begin{pmatrix}
  -f(x)& f(x)
  \\0&0
\end{pmatrix},$
where $f(x)$ is a smooth positive function with compact support and
$f(x)=x$ for $x\in[1, 2]$.
Let $\Delta_{12}(x)=[0,f(x))$,
$\Delta_{21}(x)=\emptyset$.
The process $\alpha(t)$ can be defined as the solution to
$$d\alpha(t)=\int_{\R}h(X(t),\alpha(t-), z)\p(dt, dz)$$	
where
$h(x, 1, z)=\1_{\{z\in\Delta_{12}(x)\}}$ and $h(x, 2, z)=0,$
$\p(dt, dz)$ is a Poisson random measure with intensity $dt\times\m(dz)$ and $\m$ is the Lebesgue measure on $\R$.

 Let $y>0$ and  $(X^{y, 1}(t),\alpha^{y,1}(t))$ be the solution with initial data $(y, 1)$.  Let
$\tau^y=\inf\{t\geq0: \alpha^{y,1}(t)=2\}.$
For $x\in[0,1]$ we have
$X^{1+x,1}(t)=1+x,\ \forall\,t\in[0,\tau^{1+x}]$, and $\alpha^{1+x}(t)$ stays in  state $2$ once it jumps into
$2$ since the jump intensity $q_{21}(x)=0$ for any $x\in\R$.	
Thus,
\bea \ad X^{1+x,1}(T)=1+x+T-T\wedge\tau^{x+1}
\ \hbox{and}\\
\ad X^{1,1}(T)=1+T-T\wedge\tau^{1}.\eea
As a result,
\begin{equation}\label{e0}
Z^{1+x}:=\dfrac{X^{1+x,1}(T)-X^{1,1}(T)}{x}=1+\dfrac{[T\wedge\tau^1-T\wedge\tau^{x+1}]}x
\end{equation}

If $Z^{1+x}$ converges in $L^1$ to a variable $Z_0$, then
the sequence $Z^{1+\frac{1}n}$ must be a Cauchy sequence in $L^1$.
Then it follows that,
\begin{equation}\label{e1}
\E\big|Z^{1+\frac{2}n}-Z^{1+\frac{1}n}\big|\to 0 \text{ as } n\to\infty.
\end{equation}

Now, we show that \eqref{e1} cannot be satisfied.
Let
\bea \ad Y_0=\inf\{t\geq 0: \int_0^t\int_\R\1_{\{z\in[0,1)\}}\p(ds, dz)\ne0\},\\
\ad Y_1=\inf\{t\geq 0: \int_0^t\int_\R\1_{\{z\in[1,1+\frac1n)\}}\p(ds, dz)\ne 0\},\\
\ad Y_2=\inf\{t\geq 0: \int_0^t\int_\R\1_{\{z\in[1+\frac1n,1+\frac2n)\}}\p(ds, dz)\ne 0\}.\eea
Since $[0,1), [1,1+\dfrac{1}n), [1+\frac1n,1+\frac2n)$ are disjoint sets,
we have that $Y_0, Y_1, Y_2$ are three independent exponential random variables
with parameter $1, \frac1n, \frac1n$, respectively.
Note that \begin{equation}
\label{tau-1+x}
\tau^{1+x}=\inf\{t\geq 0: \int_0^t\int_\R\1_{\{z\in[0,1+x)\}}\p(ds, dz)\ne 0\}.
\end{equation}
Thus,
$\tau^1=Y_0$, $\tau^{1+\frac1n}=Y_0\wedge Y_1$, and $\tau^{1+\frac2n}=Y_0\wedge Y_1\wedge Y_2.$
From \eqref{e0}, we have
\begin{equation}\label{e2}
\begin{aligned}
\Big|Z^{1+\frac{2}n}-Z^{1+\frac{1}n}\Big|=&\Big|n[T\wedge\tau^1-T\wedge\tau^{1+\frac1n}]-\frac{n}2[T\wedge\tau^1-T\wedge\tau^{1+\frac2n}]\Big|
\end{aligned}
\end{equation}
Let
$A=\Big\{Y_1\leq Y_2, Y_1<\dfrac{T}3, Y_0\in[\dfrac{2T}3,T]\Big\}.$
When $\omega\in A$,
$T\wedge\tau^1=\tau^1=Y_0$, $T\wedge\tau^{1+\frac1n}=T\wedge\tau^{1+\frac2n}=Y_1$.
As a result, when $\omega\in A$,
\begin{equation}\label{e2.4}
|Z^{1+\frac{2}n}-Z^{1+\frac{1}n}|=\Big|n(Y_0-Y_1)-\frac{n}2(Y_0-Y_1)\Big|=\frac{n}2(Y_0-Y_1)\geq\dfrac{Tn}6.
\end{equation}
By direct calculation,
\begin{equation}\label{e2.5}
\begin{aligned}
\PP(A)=&\PP\left\{Y_1\leq Y_2, Y_1<\dfrac{T}3\right\}\times\PP\left\{Y_0\in[\dfrac{2T}3,T]\right\}\\
=& \dfrac{1}n\left(1-\exp\Big(-(1+\dfrac{2}n)\dfrac{T}3\Big)\right)\times \Big(\exp(-\frac{2T}3)-\exp(-T)\Big)
\end{aligned}
\end{equation}
In view of \eqref{e2.4} and \eqref{e2.5},
\bea \E |Z^{1+\frac{2}n}-Z^{1+\frac{1}n}|\ad \geq \E\1_{A}|Z^{1+\frac{2}n}-Z^{1+\frac{1}n}|
\geq  \dfrac{Tn}6\PP(A)\\
\ad = \dfrac{T}6\left(1-\exp \Big(-(1+\dfrac{2}n)\dfrac{T}3\Big)\right)\times \Big(\exp (-\frac{2T}3)-\exp (-T)\Big)\\
\ad \rightarrow \dfrac{T}6\left(1-\exp \Big(-\dfrac{T}3\Big)\right)\times \Big(\exp (-\frac{2T}3)-\exp (-T)\Big)\ne 0\text{ as } n\to\infty.
\eea
As a result, $\{Z^{1+\frac{1}n}: n\in\N\}$ is not a Cauchy sequence in $L^1$.
Thus neither is it Cauchy in $L^p$, $p>1$.
 In the example, the continuous-state dependence makes the switching diffusions markedly different from that of the Markov modulated switching
diffusions.

\section{Further Properties}\label{sec:fur}
Continuing on our investigation, we derive further properties in this section.
It is arranged in two subsections.

\subsection{Uniform Continuity with Respect to Initial Data}\label{sec:con}
This section aims to obtain uniform estimates on a finite interval for the difference of two solutions
$X^{\wdt x,\al}(t) - X^{x,\al}(t)$ with distinct initial data on the continuous component.
Again, such a property is distinctly different from that of the Markov modulated switching diffusions; see Remark \ref{rem:4.2} for details.

\begin{thm}\label{prop4.2}
Assume that $b(x,i)$ and $\sigma(x,i)$ are globally Lipschitz   in $x$ $($with Lipschitz constant $\kappa)$ for each $i\in\M$,
and  that $Q(x)=(q_{ij}(x))$ is bounded and for some $\gamma_{1} > 0$ we have
\begin{equation}\label{qlip}
\sum_{k,j\in\M}|q_{kj}(x)-q_{kj}(y)|\leq K_1(1+|x|^{\gamma_1}+|y|^{\gamma_1})|x-y|,\,\forall x,y\in\R^r.
\end{equation}
Then, there exists a constant $C_T$ depending only on $T$, $K_1$, and $\kappa$
such that for any $x,\wdt x\in\R^r$, and $\al\in\M$, we have
\begin{equation}\label{e5.1}
\E \sup\limits_{t\in[0,T]}|X^{\wdt x,\al}(t) - X^{x,\al}(t)| \le C_T(1+|x|^{\gamma_1+2}+|\wdt x|^{\gamma_1+2}) |\wdt x - x|.
\end{equation} \end{thm}

\begin{proof}
By \cite[Proposition 3.5]{ZY}, for any $m\in\N$, there is a constant $C_m$ depending only on the Lipschitz constant $\kappa$
such that
\begin{equation}\label{e5.2}
\E\left(\sup\limits_{t\in[0,T]}|X^{x_0,\al_0} (t)|^m\right)\leq C_m(1+|x_0|^m)\exp(C_m T)\,\forall\, (x_0,\alpha_0)\in\R^r\times\M, T>0.
\end{equation}
Let $(X(t),\wdt  X(t), \alpha(t), \wdt \alpha(t))$ be as in the proof of Theorem \ref{thm-1}. Again, denote $\Delta = \wdt x -x$
and recall the definition of $\tau^\Delta$ in \eqref{eq:tau-d-def}.
Then
$$\sup\limits_{t\in[0,T]}|X(t) - \wdt X(t)|\leq \sup\limits_{t\in[0,T\wedge\tau^\Delta]}|X(t) - \wdt X(t)|+\sup\limits_{t\in(T\wedge\tau^\Delta,T]}|X(t) -\wdt X(t)|$$
Hence
\begin{equation}\label{e5.3}
\begin{aligned}
\E&\left(\sup\limits_{t\in[0,T]}|X(t) - \wdt X(t)|\right)\\
&\leq \E\left(\sup\limits_{t\in[0,T\wedge\tau^\Delta]}|X(t) - \wdt X(t)|\right)+\E\left(\1_{\{\tau^\Delta\leq T\}}\sup\limits_{t\in(T\wedge\tau^\Delta,T]}|X(t) - \wdt X(t)|\right).
\end{aligned}
\end{equation}
Let $\F_{T\wedge\tau^\Delta}$ be the $\sigma$-algebra generated by the processes $(X(t),\wdt  X(t), \alpha(t), \wdt \alpha(t))$
up to
the  time $T\wedge\tau^\Delta$.
We have that
\begin{equation}\label{e5.4}
\begin{aligned}
\E&\left[\1_{\{\tau^\Delta\leq T\}}\sup\limits_{t\in(T\wedge\tau^\Delta,T]}|X(t) - \wdt X(t)|\right]\\
&=\E\left[\E\Big(\1_{\{\tau^\Delta\leq T\}}\sup\limits_{t\in(T\wedge\tau^\Delta,T]}|X(t) - \wdt X(t)|\Big|\F_{T\wedge\tau^\Delta}\Big)\right]\\
&=\E\left[\1_{\{\tau^\Delta\leq T\}}\E\Big(\sup\limits_{t\in(T\wedge\tau^\Delta,T]}|X(t) - \wdt X(t)|\Big|\F_{T\wedge\tau^\Delta}\Big)\right]\\
&\leq\E\left[\1_{\{\tau^\Delta\leq T\}}\E\Big(\sup\limits_{t\in(T\wedge\tau^\Delta,T]}\{|X(t)|+|\wdt X(t)|\}\Big|\F_{T\wedge\tau^\Delta}\Big)\right]\\
&\leq\E\left[\1_{\{\tau^\Delta\leq T\}}\E\Big(\sup\limits_{t\in(T\wedge\tau^\Delta,T\wedge\tau^\Delta+T]}\{|X(t)|+\wdt X(t)|\}\Big|\F_{T\wedge\tau^\Delta}\Big)\right]\\
&\leq 2C_1e^{C_1T}\E\left[\1_{\{\tau^\Delta\leq T\}}\big(1+|X(T\wedge\tau^\Delta)|+|\wdt X(T\wedge\tau^\Delta)|\big)\right],
\end{aligned}
\end{equation}
where the last inequality follows from  the strong Markov property of $(X(t),\wdt  X(t), \alpha(t), \wdt \alpha(t))$ and \eqref{e5.2}.
Note that the coupled process $(X(t),\wdt  X(t), \alpha(t), \wdt \alpha(t))$  is the solution to \eqref{e2.3}.
For $k,l\in\M$, let $\wdt A_{kl}$ be the generator of the diffusion
\begin{equation}\label{e5.5}
\begin{cases}
dY(t)=b(Y(t), k)dt+\sigma(Y(t),k)dw(t) \\
d\wdt  Y(t)=b(\wdt Y(t), l)dt+\sigma(\wdt  Y(t),l)dw(t)
\end{cases}
\end{equation}
Then the generator $\wdt\Lom$ of $(X(t),\wdt  X(t), \alpha(t), \wdt \alpha(t))$  is given by
\begin{equation}
\label{eq:coupled-operator}
 \wdt\Lom f(x,\wdt x, k, l)=\wdt A_{kl} f(x,\wdt x,k,l)+\wdt Q(x,\wdt x) f(x,\wdt x, k, l).
\end{equation}
Define
$U(x,\wdt x, k, l)=V(k,l)(1+|x|^2+|\wdt x|^2)=\1_{\{k\ne l\}}(1+|x|^2+|\wdt x|^2)$.
Clearly, when $k=l$, $U(x,\wdt x,k,k)=0$, for all $x$ and $\wdt x$. Thus, $\wdt A_{kk} U(x,\wdt x, k, k)=0$, which combined with \eqref{e.rho} implies that \begin{equation}\label{e5.6}\begin{aligned}
\wdt\Lom U(x,\wdt x, k, k)& =\wdt Q(x,\wdt x) U(x,\wdt x, k, k)\\ & =(1+|x|^2+|\wdt x|^2)\wdt Q(x,\wdt x) V(k,k)=(1+|x|^2+|\wdt x|^2)\rho(x,\wdt x,k),
\end{aligned}\end{equation}
where $\rho(x,\wdt x,k)$ is defined in \eqref{e.rho}.
Since all moments of $X(t)$ and $\wdt X(t)$ are bounded, applying Dynkin's formula,
we have
\begin{equation}\label{e5.7}
\begin{aligned}
\E U&\big(X(T\wedge\tau^\Delta),\wdt X(T\wedge\tau^\Delta), \alpha(T\wedge\tau^\Delta), \wdt\alpha(T\wedge\tau^\Delta)\big)\\
=&\ U(x,\wdt x, \alpha, \alpha)+\E\int_0^{T\wedge\tau^\Delta}\wdt\Lom U(X(s),\wdt X(s), \alpha(s), \wdt\alpha(s))ds\\
=&\ \E\int_0^{T\wedge\tau^\Delta}(1+|X(s)|^2+|\wdt X(s)|^2)\rho(X(s),\wdt X(s),\alpha(s))ds,
\end{aligned}
\end{equation}
where
the last equality above follows from \eqref{e5.6} and the fact $\alpha(t)=\wdt\al(t)$ if $t<\tau^\Delta$.
Then  by H\"older's inequality, estimates \eqref{qlip}, and \eqref{e5.2}, we have
\begin{equation}\label{e5.8}
\begin{split}
\E&  \int_0^{T\wedge\tau^\Delta}(1+|X(t)|^2+|\wdt X(t)|^2)\rho(X(t),\wdt  X(t), \alpha(t))dt\\
& \leq \E \int_0^{T\wedge\tau^\Delta} \wdt K(1+|X(t)|^{\gamma_1+2}+|\wdt  X(t)|^{\gamma_1+2})|X(t)-\wdt  X(t)| dt\,\,\text{(for some } \wdt K>0)\\
&\leq \wdt KT\E \sup\limits_{t\leq T\wedge\tau^\Delta}\left\{(1+|X(t)|^{\gamma_1+2}+|\wdt  X(t)|^{\gamma_1+2})|X(t)-\wdt  X(t)|\right\}\\
&\leq \wdt KT\Big(\E \sup\limits_{t\leq T\wedge\tau^\Delta}\left\{(1+|X(t)|^{\gamma_1+2}+|\wdt  X(t)|^{\gamma_1+2})^{2}\right\}\Big)^{\frac{1}{2}}
\Big(\E \sup\limits_{t\leq T\wedge\tau^\Delta}\left\{|X(t)-\wdt  X(t)|)^2\right\}\Big)^{\frac{1}{2}}\\
& \leq \wdt  K_{1,T}(1+|x|^{\gamma_1+2}+|\wdt x|^{\gamma_1+2})|\Delta|\ \text{ (for some }\wdt  K_{1,T}>0),
\end{split}
\end{equation}
where the last inequality follows from \eqref{e3} and \eqref{e5.2}.

In view of \eqref{e5.4}, \eqref{e5.7}, and \eqref{e5.8},
\begin{equation}\label{e5.9}
\begin{aligned}
\E&\left(\1_{\{\tau^\Delta\leq T\}}\sup\limits_{t\in[T\wedge\tau^\Delta,T]}|X(t) - \wdt X(t)|\right)\\
&\leq
2C_1e^{C_1T}\E\left(\1_{\{\tau^\Delta\leq T\}}\big(1+|X(T\wedge\tau^\Delta)|+|\wdt X(T\wedge\tau^\Delta)|\big)\right)\\
&\leq 6C_1e^{C_1T}\E\left(\1_{\{\tau^\Delta\leq T\}}\big(1+|X(T\wedge\tau^\Delta)|^2+|\wdt X(T\wedge\tau^\Delta)|^2\big)\right)\\
&\leq 6C_1e^{C_1T}\E U\big(X(T\wedge\tau^\Delta),\wdt X(T\wedge\tau^\Delta), \alpha(T\wedge\tau^\Delta), \wdt\alpha(T\wedge\tau^\Delta)\big)\\
&\leq6C_1e^{C_1T}\wdt  K_{1,T}(1+|x|^{\gamma_1+2}+|\wdt x|^{\gamma_1+2})|\Delta|.
\end{aligned}
\end{equation}
The proof is complete by applying \eqref{e5.9} and \eqref{e3} to \eqref{e5.2}.
\end{proof}

\begin{rem}\label{rem:4.2}{\rm
In contrast to  switching diffusions with Markovian switching, in which the switching is independent of the Brownian motion,
in our case,  we have to estimate $X(t)-\wdt X(t)$ after $\tau^\Delta$ if $\tau^\Delta\leq T$. Note that the difference
$|X(t)-\wdt X(t)|$ after $\tau^\Delta$ cannot be estimated by something related to $|x-\wdt x|$
since after $\tau^\Delta$, the evolutions of $X(t)$ and $\wdt X(t)$ are quite different.
We can only estimate $|X(t)-\wdt X(t)|$ by $|X(t)|+|\wdt X(t)|$ for $t>\tau^\Delta$.
For this reason, in general,
it does not seem that  $C_T(1+|x|^{\gamma_1+2}+|\wdt x|^{\gamma_1+2})$ in \eqref{e5.1}
can be replaced with a
 constant $K$ independent of $x$  and $\wdt x$,
even under the condition that $Q(x)$ is globally Lipschitz.
However, if we assume $b$ and $\sigma$ are bounded,
$C_T(1+|x|^{\gamma_1+2}+|\wdt x|^{\gamma_1+2})$ in \eqref{e5.1} can be replaced by a constant $K$.
Finally, we remark that if the condition for $Q(x)$ is reduced  to
$$ \sum_{k,j\in\M}|q_{kj}(x)-q_{kj}(y)|\leq K_1(1+|x|^{\gamma_1}+|y|^{\gamma_1})|x-y|^\lambda,\ \hbox{ for all }\  x,y\in\R^r \ \hbox{ and some }\ 0<\lambda \le 1 $$
then we can obtain
\begin{equation*}
\E \sup\limits_{t\in[0,T]}|X^{\wdt x,\al}(t) - X^{x,\al}(t)| \le C_T(1+|x|^{\gamma_1+2}+|\wdt x|^{\gamma_1+2}) |\wdt x - x|^\lambda.
\end{equation*}
}
\end{rem}

\subsection{Smoothness of A Functional of the Switching Diffusion}\label{sec:fun}
Continuing on our investigations, this section obtains smoothness of nonlinear functional of the switching diffusions.
Once again, the continuous-state dependent switching presents much difficulty. Our task is to untangle the
process $\alpha^{x,\alpha}(t)$ and $\alpha^{\wdt x,\alpha}(t)$ as in the previous sections.
As alluded to in Remark \ref{rem:4.2}, the difficulty due to the continuous state dependent switching is particularly pronounced.

\begin{thm}\label{thm5.2} For each $i\in\M$,
assume that $b(x,i)$ and $\sigma(x,i)$ are globally Lipschitz  in $x$ $($with Lipschitz constant $\kappa)$,
that  $b(\cdot,i),\sigma(\cdot,i),  q_{ij}(\cdot)\in C^2$,
 that there is
a $\phi(\cdot,i) \in C^2$,  and that
\begin{equation}\label{e.5bd}
|D^\beta_xb(x,i)|+|D^\beta_x\sigma(x,i)|+|D^\beta_x\phi(x,i)|\leq K(1+|x|^{\gamma}), i\in\M, |\beta|\leq 2.
\end{equation}
for some positive constant $\gamma$.
Assume further that $|D^\beta q_{ij}(\cdot)|$ are Lipschitz and bounded uniformly by some constant $M$ for $|\beta|\leq 2$.
Then, $u(t,x,i)=\E[\phi(X^{x,i}(t),\alpha^{x,i}(t))]$
is twice continuously differentiable with respect to the variable $x$.
\end{thm}

\begin{proof}
Again, for simplicity, we work out the details for the one-dimensional case.
Denote by $ (\partial / \partial x) \phi(\cdot,i)$ for each $i\in \M$, and $ (d/dx) q_{ij}(\cdot)$ the derivatives of $\phi(\cdot,i)$ and $q_{ij}(\cdot)$ with respect to $x$, respectively.
Let $\chi(t)$ be the Markov chain in $\M$ with generators
$\hat q_{ij}=1$ if $i\ne j$, $\hat q_{ii}=-m_0+1$.
Let $\theta^{i_0}_n$ be the $n$-th jump time of $\chi(t)$ given that $\chi(0)=i_0$.
Let $Z^{x_0,i_0}(t)$ be the solution with initial value $Z^{x_0,i_0}(0)=x_0$ to
\begin{equation}\label{e4.3}
dZ(t)=b(Z(t), \chi(t))dt+\sigma(Z(t),\chi(t))dw(t).\end{equation}
Let $\tau^{x_0,i_0}_n$ be the $n$-th jump time of $\alpha^{x_0,i_0}(t)$.
By a change of measure (see \cite{EF-1}),
we have
\begin{equation}\label{e4.4}
\begin{aligned}
\E\Big[ \phi&(X^{x_0,i_0}(T),\alpha^{x_0,i_0}(T))\1_{\{\tau^{x_0,i_0}_n\leq T<\tau^{x_0,i_0}_{n+1}\}}\prod_{k=1}^n\1_{\{\alpha^{x_0,i_0}(\tau_{k}^{x_{0},i_{0}})=i_k\}}\Big]\\
=&\exp((m_0-1)T)\E\Big[ \phi(Z^{x_0,i_0}(T),i_n)\1_{\{\theta^{i_0}_n\leq T<\theta^{i_0}_{n+1}\}}\Big(\prod_{k=1}^n\1_{\{\chi(\theta^{i_0}_{k})=i_k\}}\Big)\\
&\qquad\qquad\qquad\qquad\quad\times \exp\Big\{-\int_{0}^Tq_{\chi(s)}(Z^{x_0,i_0}(s))ds\Big\}\prod_{k=0}^{n-1}\Big(q_{i_k i_{k+1}}
(Z^{x_0,i_0}({\theta^{i_{0}}_{k+1}}))\Big)\Big],
\end{aligned}
\end{equation}
where $q_i (x)= \sum_{j\not = i} q_{ij} (x)$.
Now, fix $x$ and $\alpha=i_0$
and let
 $(X(t),\alpha(t))$,  $(\wdt  X(t),\wdt  \alpha(t))$ be the switching-diffusion processes satisfying \eqref{sde} and \eqref{tran} with
initial conditions $(x,\alpha)$ and $(\wdt  x,\alpha)$, respectively, where $\wdt x=x+\Delta, \alpha=i_0$. Likewise,
let $Z(t), \wdt Z(t)$ be the solutions to \eqref{e4.3} with initial conditions
$x,\wdt x$ respectively.
By standard arguments (e.g. \cite[Lemma 3.3]{MaoY} and \cite[Lemma 3.2]{YZ}), we have
\begin{equation}\label{e4.4a}
\begin{cases}
\E\sup\limits_{t\in[0,T]}|\wdt Z(t)-Z(t)|^4\leq K\Delta^4\text{ for some }K>0,\\
\E\sup\limits_{t\in[0,T]}|Z(t)|^m\leq K_m(1+|x|^m) \text{ for any } m>0.
\end{cases}
\end{equation}
Adapting the proofs of \cite[Theorems 5.5.2, 5.5.3]{AF} with a slight modification in replacing fixed times by stopping times,
we can show that
\begin{equation}\label{e4.4b}
\mu_\Delta:=\sup\limits_{\{\vartheta\in\mathcal T\}}\dfrac{\E |\wdt Z(\vartheta)-Z(\vartheta)-\eta(\vartheta)\Delta|^2}{\Delta^2}\to0\text{ as } \Delta\to0,
\end{equation}
where $\eta(t)$ is the derivative of $Z(t)$ with respect to the initial condition $x$ which exists in $L^2$-sense
and
$\mathcal T$ is the set of all stopping times that are bounded above by $T$ almost surely.
Note that $\eta(t)$ satisfies
$$\eta(t)=1+\int_0^t\eta(s)b_x(Z(s),\chi(s))ds+\int_0^t\eta(s)\sigma_x(Z(s),\chi(s))dw(s).$$
Using the above equation, \eqref{e4.4a}, the Burkholder-Davis-Gundy inequality, Gronwall's inequality and standard arguments in \cite[Lemma 3.2]{YZ} or \cite[Theorem 5.2.2]{AF},
we can show that
\begin{equation}\label{e4.4c}
\E\sup\limits_{t\in[0,T]}|\eta(t)|^2\leq C(T,x),
\end{equation}
where $C(T,x)$ is a constant depending on $T$ and $x$.
By Taylor's expansion, it can be easily shown that
\begin{equation}\label{e4.5}
\begin{aligned}
\bigg|&\exp\Big\{-\int_{0}^Tq_{\chi(s)}(\wdt Z_s)ds\Big\}-\exp\Big\{-\int_{0}^Tq_{\chi(s)}(Z_s)ds\Big\}\\
&\qquad +\left(\int_{0}^{T}(\wdt Z(s)-Z(s)){d\over dz}q_{\chi(s)}(Z(s))
ds\right)\exp\Big\{-\int_{0}^{T}q_{\chi(s)}(Z(s))ds\Big\}\bigg|\\
&\leq M\int_{0}^{T}|\wdt Z(s)-Z(s)|^2ds.
\end{aligned}
\end{equation}In view of \eqref{e4.5}
\begin{equation}\label{e4.6}
\begin{aligned}
\bigg|&\dfrac1\Delta\left(\exp\Big\{-\int_{0}^Tq_{\chi(s)}(\wdt Z_s)ds\Big\}-\exp\Big\{-\int_{0}^Tq_{\chi(s)}(Z_s)ds\Big\}\right)\\
&\qquad +\left(\int_{0}^{T}\eta(s)
{d \over dz} q_{\chi(s)}
(Z(s))ds\right)\exp\Big\{-\int_{0}^{T}q_{\chi(s)}(Z(s))ds\Big\}\bigg|\\
&\leq
\dfrac1\Delta\bigg|\exp\Big\{-\int_{0}^Tq_{\chi(s)}(\wdt Z_s)ds\Big\}-\exp\Big\{-\int_{0}^Tq_{\chi(s)}(Z_s)ds\Big\}\\
&\quad\qquad+\left(\int_{0}^{T}(\wdt Z(s)-Z(s)){d\over dz}q_{\chi(s)}(Z(s))
ds\right)\exp\Big\{-\int_{0}^{T}q_{\chi(s)}(Z(s))ds\Big\}\bigg|\\
&\quad+\bigg|\left(\int_{0}^{T}\Big[\eta(s)-\dfrac{\wdt Z(s)-Z(s)}\Delta\Big]{d\over dz}q_{\chi(s)}(Z(s))
ds\right)\exp\Big\{-\int_{0}^{T}q_{\chi(s)}(Z(s))ds\Big\}\bigg|\\
&\leq
\dfrac{M}\Delta\int_{0}^{T}\left(|\wdt Z(s)-Z(s)|^2+|\wdt Z(s)-Z(s)-\Delta\eta(s)|\right)ds.\end{aligned}
\end{equation}
We  have from Taylor's expansion
\begin{equation}\label{e4.7} \begin{aligned} \bigg|&\dfrac{q_{i_ki_{k+1}}(\wdt Z(\theta_{k+1}))-q_{i_ki_{k+1}}( Z(\theta_{k+1}))}\Delta-\eta(\theta_{k+1}) {d \over dz} q_{i_ki_{k+1}}( Z(\theta_{k+1})) \bigg|\\ &\leq \dfrac1\Delta\left|q_{i_ki_{k+1}}(\wdt Z(\theta_{k+1}))-q_{i_ki_{k+1}}( Z(\theta_{k+1}))- \Big(\wdt Z(\theta_{k+1})-Z(\theta_{k+1})\Big) {d \over dz} q_{i_ki_{k+1}}( Z(\theta_{k+1})) \right|\\ &\quad+\dfrac1\Delta\left|\wdt Z(\theta_{k+1})-Z(\theta_{k+1})-\Delta\eta(\theta_{k+1})\right|\left|{d \over dz} q_{i_ki_{k+1}}( Z(\theta_{k+1})) \right|\\ &\leq \dfrac{M}\Delta\left(|\wdt Z(\theta_{k+1})-Z(\theta_{k+1})|^2+|\wdt Z(\theta_{k+1})-Z(\theta_{k+1})-\Delta\eta(\theta_{k+1})|\right). \end{aligned} \end{equation}Similar to \cite[(5.22) on p.123]{AF},
we  have
\begin{equation}\label{e4.8}
\begin{aligned}
\psi_\Delta:=\bigg|&\dfrac{\phi(\wdt Z(T),\chi(T))-\phi(Z(T),\chi(T))}\Delta-\eta(T)
{\partial \over \partial z} \phi(Z(T),\chi(T))
\bigg|\to 0 \text{ in } L^2(\Omega) \text{ as }\Delta\to0.
\end{aligned}
\end{equation}

To simplify  notation, we drop the superscript of $\theta_n^{i_0}$ and for each $n\in\N$,
let $$ A_n:=\{(i_1,\dots,i_{n+1}): i_k\in\M, i_k\ne i_{k-1},k=1,\dots,n+1\}.$$
For each $\I=(i_1,\dots,i_{n+1})\in A_n$,
let
$$f_\I:=\1_{\{\theta_n\leq T<\theta_{n+1}\}}\Big(\prod_{k=1}^n\1_{\{\chi(\theta_k)=i_k\}}\Big).$$
We compute
\begin{equation}\label{e4.9}
\begin{aligned}
\Bigg|\dfrac1\Delta\Big[ &\phi(\wdt Z(T),i_n)f_\I \exp\Big\{-\int_{0}^{T}q_{\chi(s)}(\wdt Z(s))ds\Big\}\prod_{k=0}^{n-1}q_{i_ki_{k+1}}
(\wdt Z(\theta_{k+1}))\\
&-\phi( Z(T),i_n)f_\I \exp\Big\{-\int_{0}^{T}q_{\chi(s)}(Z(s))ds\Big\}\prod_{k=0}^{n-1}q_{i_ki_{k+1}}
(Z(\theta_{k+1}))\Big]\\
& -\eta(T)
{\partial \over \partial z}\phi(Z(T),i_n)
f_\I \exp\Big\{-\int_{0}^{T}q_{\chi(s)}(Z(s))ds\Big\}\prod_{k=0}^{n-1}q_{i_ki_{k+1}}
(Z(\theta_{k+1}))\\
& -\sum_{j=0}^{n-1}\phi( Z(T),i_n)f_\I \exp\Big\{-\int_{0}^{T}q_{\chi(s)}(Z(s))ds\Big\}\eta(\theta_{j+1})
{d\over dz} q_{i_ji_{j+1}}(Z(\theta_{j+1}))\\
& \quad \hfill \times\prod_{k=0,k\ne j}^{n-1}q_{i_ki_{k+1}}
(Z(\theta_{k+1}))\\
& +\phi( Z(T),i_n)f_\I \exp\Big\{-\int_{0}^{T}q_{\chi(s)}(Z(s))ds\Big\}\Big(\int_{0}^{T}\eta(s)
{d\over dz} q_{\chi(s)}(Z(s))
ds\Big) \\
& \quad \hfill \times \prod_{k=0}^{n-1}q_{i_ki_{k+1}}
(Z(\theta_{k+1}))\Bigg|\\
&\!\!\!\!\leq Cf_\I M^n\bigg[\psi_\Delta+(1+|\wdt Z(T)|^\gamma+|Z(T)|^\gamma)\int_0^T\dfrac{|\wdt Z(s)-Z(s)|^2+|\wdt Z(s)-Z(s)-\Delta\eta(s)|}\Delta ds\\
& +(1+|\wdt Z(T)|^\gamma+|Z(T)|^\gamma)\sum_{j=0}^n\dfrac{|\wdt Z(\theta_{j+1})-Z(\theta_{j+1})|^2+|\wdt Z(\theta_{j+1})-Z(\theta_{j+1})-\Delta\eta(\theta_{j+1})|}\Delta\biggr].
\end{aligned}
\end{equation}
Let $\hat\zeta_n(T)$ be defined by
$$
\begin{aligned}
\hat\zeta_n(T)=&\sum_{\I\in A_n}\bigg[\eta(T)
{\partial \over \partial z}\phi(Z(T),i_n)
f_\I \exp\Big\{-\int_{0}^{T}q_{\chi(s)}(Z(s))ds\Big\}\prod_{k=0}^{n-1}q_{i_ki_{k+1}}
(Z(\theta_{k+1}))\\
&+\sum_{j=0}^{n-1}\phi( Z(T),i_n)f_\I \exp\Big\{-\int_{0}^{T}q_{\chi(s)}(Z(s))ds\Big\}\eta(\theta_{j+1})
{d\over dz}q_{i_ji_{j+1}}(Z(\theta_{j+1}))\\
&\qquad\hfill \times
\prod_{k=0,k\ne j}^{n-1}q_{i_ki_{k+1}}
(Z(\theta_{k+1}))\\
&-\phi( Z(T),i_n)f_\I \exp\Big\{-\int_{0}^{T}q_{\chi(s)}(Z(s))ds\Big\}\Big(\int_{0}^{T}\eta(s)
{d\over dz} q_{\chi(s)}(Z(s))
ds\Big)\\
& \qquad \hfill \times \prod_{k=0}^{n-1}q_{i_ki_{k+1}}
(Z(\theta_{k+1}))\bigg].
\end{aligned}
$$ Note that $\sum_{\I\in A_n} f_{\I} =1$.
It then follows from \eqref{e4.9} that for each $n\in\N$,
$\hat\zeta_n(t)$ such that
\begin{equation}\label{e4.10}
\begin{aligned}
\E\sum_{\I\in A_n}&\Bigg|\dfrac1\Delta\Big[ \phi(\wdt Z(T),i_n)f_\I\exp\Big\{-\int_{0}^{T}q_{\chi(s)}(\wdt Z(s))ds\Big\}\prod_{k=0}^{n-1}q_{i_ki_{k+1}}
(\wdt Z(\theta_{k+1}))\\
&-\phi( Z(T),i_n)f_\I \exp\Big\{-\int_{0}^{T}q_{\chi(s)}( Z(s))ds\Big\}\prod_{k=0}^{n-1}q_{i_ki_{k+1}}
(Z(\theta_{k+1}))\Big]-\hat\zeta_n(T)\Bigg|\\
\leq& M^n\E \1_{\{\theta_n\leq T<\theta_{n+1}\}}\psi_\Delta\\
&+\dfrac{C}\Delta M^n\E\1_{\{\theta_n\leq T<\theta_{n+1}\}}(1+|\wdt Z(T)|^\gamma+|Z(T)|^\gamma)\\
& \qquad \qquad \times \int_0^T\Big(|\wdt Z(s)-Z(s)|^2+|\wdt Z(s)-Z(s)-\Delta\eta(s)|\Big)ds\\
&+\dfrac{C}\Delta M^n\E\1_{\{\theta_n\leq T<\theta_{n+1}\}}(1+|\wdt Z(T)|^\gamma+|Z(T)|^\gamma)\\
&\qquad\qquad\times\sum_{j=0}^n\left(|\wdt Z(\theta_{j+1})-Z(\theta_{j+1})|^2+|\wdt Z(\theta_{j+1})-Z(\theta_{j+1})-\Delta\eta(\theta_{j+1})|\right).
\end{aligned}
\end{equation}
By H\"older's inequality and the fact that
\begin{equation}\label{e4.10a}
\PP\{\theta_n\leq T<\theta_{n+1}\}=\dfrac{\exp(-m_0T)(m_0T)^n}{n!},
\end{equation}
we can obtain the  following
estimate:
\begin{equation}\label{e4.11}
\E \1_{\{\theta_n\leq T<\theta_{n+1}\}}\psi_\Delta\leq \Big(\PP\{\theta_n\leq T<\theta_{n+1}\}\E\psi^2_\Delta\Big)^{\frac12}
\leq \left(\dfrac{\exp(-m_0T)(m_0T)^n}{n!}\right)^{\frac12}(\E\psi^2_\Delta)^{\frac12}.
\end{equation}
Similarly, with the aid of H\"older's and Minkowski's inequalities, we have
\begin{equation}
\begin{aligned}
 \dfrac{1}\Delta& \E\1_{\{\theta_n\leq T<\theta_{n+1}\}}(1+|\wdt Z(T)|^\gamma+|Z(T)|^\gamma)\int_0^T\Big(|\wdt Z(s)-Z(s)|^2+|\wdt Z(s)-Z(s)-\Delta\eta(s)|\Big)ds\\
 \leq&  \Big(\PP\{\theta_n\leq T<\theta_{n+1}\}\Big)^{\frac14}\Big(1+(\E|\wdt Z(T)|^{4\gamma})^{\frac14}+(\E|Z(T)|^{4\gamma})^{\frac14}\Big)\\
 &\times\dfrac1\Delta\left[\int_0^T\Big(\E|\wdt Z(s)-Z(s)|^4\Big)^{\frac12} ds+\int_0^T\Big(\E|\wdt Z(s)-Z(s)-\Delta\eta(s)|^2\Big)^{\frac12} ds\right)\\
 \leq&K \left(\dfrac{\exp(-m_0T)(m_0T)^n}{n!}\right)^{\frac14}(\Delta+\mu_\Delta^{\frac12})  \text{ (by \eqref{e4.4a} and \eqref{e4.4b}) }
\end{aligned}\label{e4.12}
\end{equation}
Likewise,
\begin{equation}\label{e4.13}
\begin{aligned}
\dfrac{1}\Delta \E\bigg[&\1_{\{\theta_n\leq T<\theta_{n+1}\}}(1+|\wdt Z(T)|^\gamma+|Z(T)|^\gamma)\\
&\times\sum_{j=0}^n\left(|\wdt Z(\theta_{j+1})-Z(\theta_{j+1})|^2+|\wdt Z(\theta_{j+1})-Z(\theta_{j+1})-\Delta\eta(\theta_{j+1})|\right)\bigg]\\
\leq&nK \left(\dfrac{\exp(-m_0T)(m_0T)^n}{n!}\right)^{\frac14}(\Delta+\mu_\Delta^{\frac12}).
\end{aligned}
\end{equation}
Applying \eqref{e4.11}, \eqref{e4.12} and \eqref{e4.13} to \eqref{e4.10} to obtain
\begin{equation}\label{e4.14}
\begin{aligned}
\E\sum_{\I\in A_n}&\Bigg|\dfrac1\Delta\Big[ \phi(\wdt Z(T),i_n)f_\I\exp\Big\{-\int_{0}^{T}q_{\chi(s)}(\wdt Z(s))ds\Big\}\prod_{k=0}^{n-1}q_{i_ki_{k+1}}
(\wdt Z(\theta_{k+1}))\\
&-\phi( Z(T),i_n)f_\I \exp\Big\{-\int_{0}^{T}q_{\chi(s)}( Z(s))ds\Big\}\prod_{k=0}^{n-1}q_{i_ki_{k+1}}
(Z(\theta_{k+1}))\Big]-\hat\zeta_n(T)\Bigg|\\
&\leq (n+1)\wdt K M^n\left(\dfrac{\exp(-m_0T)(m_0T)^n}{n!}\right)^{\frac14}(\mu_\Delta^{\frac12}+(\E\psi^2_\Delta)^{\frac12}+\Delta)
\end{aligned}
\end{equation}
for some $\wdt K$ independent of $n$.
It is not difficult to show that
\begin{equation}\label{e4.14a}
\sum_{n=0}^\infty (n+1)\wdt K M^n\left(\dfrac{\exp(-m_0T)(m_0T)^n}{n!}\right)^{\frac14}<\infty.
\end{equation}
Moreover,  by virtue of \eqref{e4.4b} and \eqref{e4.8}, we have
$$\lim\limits_{\Delta\to0}\Big(\mu_\Delta^{\frac12}+(\E\psi^2_\Delta)^{\frac12}+\Delta\Big)=0.$$
Thus,
\begin{equation}\label{e4.15}
\begin{aligned}
\E\sum_{n=0}^\infty \Bigg|\sum_{\I\in A_n}& \dfrac1\Delta\Big[  \phi(\wdt Z(T),i_n)f_\I\exp\Big\{-\int_{0}^{T}q_{\chi(s)}(\wdt Z(s))ds\Big\}\prod_{k=0}^{n-1}q_{i_ki_{k+1}}
(\wdt Z(\theta_{k+1}))\\
&- \phi( Z(T),i_n)f_\I\exp\Big\{-\int_{0}^{T}q_{\chi(s)}( Z(s))ds\Big\}\prod_{k=0}^{n-1}q_{i_ki_{k+1}}
( Z(\theta_{k+1}))\Big]-\hat\zeta_n(T)\Bigg|\to0
\end{aligned}
\end{equation} as  $\Delta\to0$.

Similar to \eqref{e4.14} and \eqref{e4.14a}, by using Holder's inequality, \eqref{e4.4c},
and \eqref{e4.10a}, we can obtain that
\begin{equation}\label{e4.16}
\begin{aligned}
\E \biggl|\sum_{n=0}^\infty \hat\zeta_n(T)\biggr|\leq \hat C(T,x)<\infty.
\end{aligned}
\end{equation}
As a result of \eqref{e4.15} and \eqref{e4.16}, we deduce
\begin{equation}\label{e4.17}
\begin{aligned} \frac{1}{\Delta} & [u(T,x+\Delta, i) - u(T,x,i)]\\ & =
\dfrac1\Delta \E\Big(\phi(\wdt X(T),\wdt\alpha(T))-\phi(X(T),\alpha(T))\Big)\\
& =
\dfrac1\Delta\E\sum_{n=0}^\infty\sum_{\I\in A_n}\Big[ \phi(\wdt Z(T),i_n)f_\I\exp\Big\{-\int_{0}^{T}q_{\chi(s)}(\wdt Z(s))ds\Big\}\prod_{k=0}^{n-1}q_{i_ki_{k+1}}
(\wdt Z(\theta_{k+1}))\\
&\qquad\qquad\qquad -\phi(Z(T),i_n)f_\I\exp\Big\{-\int_{0}^{T}q_{\chi(s)}( Z(s))ds\Big\}\prod_{k=0}^{n-1}q_{i_ki_{k+1}}
(Z(\theta_{k+1}))
\Big]\\
&\to\E \sum_{n=0}^\infty \hat\zeta_n(T)\text{ as }\Delta\to0.
\end{aligned}
\end{equation}
We have therefore  proved that $u(t,x,i)$ is differentiable with respect to $x$ with $$\frac{\partial}{\partial x} u(t,x,i)  = \E_{x,i}\sum_{n=0}^\infty \hat\zeta_n(T).$$
With Taylor's expansion up to the second order  and using the same method,
we can show that $u(t,x,i)$ is twice differentiable with respect to $x$.
\end{proof}

\begin{rem}
{\rm  If \eqref{e.5bd}
holds for any $\beta\leq n$,
then we can use a change of measure argument,
 Taylor's expansion,
and arguments in the proof of Theorem \ref{thm5.2}
to obtain differentiability up to order $n$ of
$u(t,x,i)$.
However, the estimates would be more complicated.
}
\end{rem}

\section{Feller Property under Non-Global Lipschitz Condition}\label{sec:fel}
This section is devoted to obtaining Feller properties of switching diffusions under non-Lipschitz condition.
There has been much work on Feller properties of switching diffusions in the literature. However, to the best of our knowledge, all of the work up to date has been concentrated on the case under a global Lipschitz condition; see for example,
\cite{Xi08,Xi09,YinZ}. When the global Lipschitz condition is violated, will the processes still possess Feller property?
We address this issue in what follows.

\begin{thm}\label{Feller}
Assume that the hypothesis of Theorem \ref{exist-unique} is satisfied.
Then the solution process $(X(t),\alpha(t))$
for the system given by \eqref{sde} and \eqref{tran}
is a Markov-Feller process.
\end{thm}

\begin{proof}
The Markov property follows from standard arguments.
In what follows, we focus on
the proof of the Feller property. First, we suppose that
$b(x,i)$ and $\sigma(x,i)$ are Lipschitz  in $x$ for each $i\in\M$.
In fact, the Feller property was obtained in \cite[Section 2.5]{YinZ} under global Lipschitz condition. The proof was   rather long.
Here, using our results of the current paper, we provide an alternative proof.
Let $f(\cdot,\cdot):\R^r\times\M\mapsto\R$ be a bounded and continuous function.
Fix $(x,\alpha)\in \R^r\times\M$ and $t>0$. Let $\{x_n\}$ be any sequence converging to $x$ as $n\to\infty$.
By the definition of the limit superior of a sequence,
since $f(\cdot,\cdot)$ is bounded,
we can always extract a subsequence
$\{\E f(X^{x_{n_k},\alpha}(t),\alpha^{x_{n_k},\alpha}(t))\}$ from $\{\E f(X^{x_{n},\alpha}(t),\alpha^{x_{n},\alpha}(t))\}$ such that
\begin{equation}\label{e.xnk}
\lim_{k\to\infty}\E f(X^{x_{n_k},\alpha}(t),\alpha^{x_{n_k},\alpha}(t))=\limsup_{n\to\infty}\E f(X^{x_{n},\alpha}(t),\alpha^{x_{n},\alpha}(t)).
\end{equation}
In view of \eqref{e4} and \eqref{e7},
$(X^{x_{n_k},\alpha}(t),\alpha^{x_{n_k},\alpha}(t))$
converges  to $(X^{x,\alpha}(t),\alpha^{x,\alpha}(t))$ in probability as $k\to \infty$. Then it has a subsequence converging almost surely to
$(X^{x,\alpha}(t),\alpha^{x,\alpha}(t))$. Without loss of generality,
we may assume that $(X^{x_{n_k},\alpha}(t),\alpha^{x_{n_k},\alpha}(t))$ converges almost surely to $(X^{x,\alpha}(t),\alpha^{x,\alpha}(t))$ as $k\to\infty$.
Then the dominated convergence theorem and \eqref{e.xnk} imply that
$$\limsup_{n\to\infty}\E f(X^{x_{n},\alpha}(t),\alpha^{x_{n},\alpha}(t))=\lim_{k\to\infty}\E f(X^{x_{n_k},\alpha}(t),\alpha^{x_{n_k},\alpha}(t))=\E f(X^{x,\alpha}(t),\alpha^{x,\alpha}(t)).$$
Likewise,
$$\liminf_{n\to\infty}\E f(X^{x_{n},\alpha}(t),\alpha^{x_{n},\alpha}(t))=\E f(X^{x,\alpha}(t),\alpha^{x,\alpha}(t)).$$
Thus, we  obtain the Feller property under the condition that $b(x,i)$ and $\sigma(x,i)$ are Lipschitz  in $x$ for each $i\in\M$.

Next we relax the condition and assume only the local Lipschitz continuity as in the statement of the theorem.
It is  proved in \cite[Theorem 2.7]{YinZ} that
for any $R>0$, $\eps>0$, and $t>0$, there is an $H_R>0$ such that
\begin{equation}\label{e-regular}
\PP\{|X^{x,\alpha}(s)|<H_R \text{ for all } s\in[0,t]\}>1-\eps \text{ if } |x|\leq R.
\end{equation}
Now fix $(x,\alpha)\in \R^r\times\M$ and $R>|x|+1$.
Let $f(\cdot,\cdot):\R^r\times\M\mapsto\R$ be a  continuous function satisfying $|f(x,\alpha)|\leq 1$ for all  $(x,\alpha) \in \rr^r \times \M$.
Let $\psi(\cdot):\R^r\mapsto[0,1]$ be a smooth function with compact support satisfying $\psi(x)=1$ if $|x|\leq H_R$.
By the first part of this proof, the process $(\hat X(t),\hat \alpha(t))$ satisfying
\begin{equation}\label{sde-hat}
\begin{cases}
d\hat X(t)=\psi(\hat X(t))b(\hat X(t), \hat \alpha(t))dt+\psi(\hat X(t))\sigma(\hat X(t),\hat \alpha(t))dw(t),\\
\PP\{\hat\alpha(t+\Delta)=j|\hat\alpha(t)=i, \hat X(s),\hat \alpha(s), s\leq t\}=q_{ij}(\hat X(t))\Delta+o(\Delta) \text{ if } i\ne j
\end{cases}
\end{equation}
has the Feller property.
Thus, there exists some $\delta\in(0,1)$ such that
\begin{equation}\label{feller-xhat}
\Big|\E f(\hat X^{x+h, \alpha}(t),\hat \alpha^{x+h,\alpha}(t))-\E f(\hat X^{x,\alpha}(t),\hat \alpha^{x,\alpha}(t))\Big|<\eps\text{ for any } h\in\R^r, |h|\leq\delta.
\end{equation}
where $(\hat X^{x, \alpha}(t),\hat \alpha^{x,\alpha}(t))$ denotes the solution of \eqref{sde-hat} with initial value $(x,\alpha)$.
By the definition of $\psi(x)$ and \eqref{e-regular},
we have that
\begin{equation}\label{e-x-xhat}
\PP\{X^{x+h,\alpha}(t)=\hat X^{x+h,\alpha}(t), \alpha^{x+h,\alpha}(t)=\hat \alpha^{x+h,\alpha}(t)\}>1-\eps \text{ if } |h|\leq 1.
\end{equation}
In view of \eqref{feller-xhat} and \eqref{e-x-xhat} and the assumption that $|f(x,\alpha)|\leq 1$ for all $(x,\alpha) \in \R^r\times \M$,
we  obtain
\begin{equation}\label{feller-x}
\Big|\E f( X^{x+h, \alpha}(t), \alpha^{x+h,\alpha}(t))-\E f( X^{x,\alpha}(t), \alpha^{x,\alpha}(t))\Big|<7\eps\text{ for any } h\in\R^r, |h|\leq\delta.
\end{equation}
The Feller property is therefore proved.
\end{proof}

\section{An Example}\label{sec:lot}
As an application of the well-posedness properties studied in the previous sections, this section deals with  a competitive Lotka-Volterra system with regime switching.
 Such a model and many of its variants were extensively investigated in the literature;
 we refer the reader to \cite{ZY09}
  and many references therein for the recent developments.

\begin{exm}\rm
Consider a stochastic competitive Lotka-Volterra model with regime switching
\begin{equation}\label{e1-ex1}
dX_i(t)=X_i(t)\left[b_i(\alpha(t))-\sum_{j=1}^ra_{ij}(\alpha(t))X_j(t)\right]dt+X_i(t)\sigma_{i}(\alpha(t))dW_i(t),\, i=1,\dots,r,
\end{equation}
where $b_i(\cdot), \sigma_i(\cdot), i\in\{1,\dots,r\}, a_{ij}(\cdot), i,j\in\{1,\dots,r\}$ are functions from $\M$ to $\R$
and $a_{ii}(k)>0, a_{ij}(k)\geq0, i,j\in\{1,\dots,r\}, k\in\M$, $W_i(t), i\in\{1,\dots,r\}$ are Brownian motions,
$\alpha(t)$ is the switching process taking value in $\M=\{1,\dots,m_0\}$
with generators $Q(x)=(q_{ij}(x))_{m_0\times m_0}$.
Assume that $q_{ij}(\cdot), i, j\in\M$ are bounded and continuous.
\eqref{e1-ex1} can be written in the matrix form
\begin{equation}\label{e2-ex1}
dX(t)=\diag(X(t))\left[b(\alpha(t))-A(\alpha(t))X(t)\right]dt+\diag(X(t))\diag(\sigma(\alpha(t))dW(t),
\end{equation}
where $b(k)=(b_1(k),\dots,b_r(k))$, $A(k)=(a_{ij}(k))_{r\times r}$, $\sigma(k)=(\sigma_1(k),\dots,\sigma_r(k))$
and
$W(t)=(W_1(t),\dots,W_r(t))$.
The model \eqref{e1-ex1} with Markovian switching, that is, when  $Q(x)$ is a constant matrix, was studied in \cite{ZY09}.
Although we are considering a more complex model with state-dependent switching, following the proofs of \cite[Theorems 2.1, 3.1]{ZY09}, we can still obtain that
\begin{itemize}
\item For any $x\in\R^{r,\circ}_+:=\{(x_1,\dots,x_r): x_i>0, i=1,\dots,r\}$,
there exists  a unique   global solution $(X^{x,\alpha}(t),\alpha^{x,\alpha}(t))$ with $X^{x,\alpha}(t)=(X_1^{x,\alpha}(t),\dots, X_r^{x,\alpha}(t))$
to \eqref{e1-ex1} with initial value $x$. Moreover,
$$\PP\{X^x(t)\in\R^{r,\circ}_+\,\forall\, t\geq0\}=1.$$
\item For any $m>0$, there exists a constant $K_m>0$  such that
\begin{equation}\label{e3-ex1}
\E|X^x(t)|^m\leq K_m(1+|x|^m)\,\text{ for all } x\in\R^{n,\circ}_+, t\geq0,
\end{equation}
where we use the norm $|x|=\sum_{i=1}^r|x_i|$ for $x=(x_1,\dots,x_r)\in\R^r$.
\end{itemize}
We aim to show that the model \eqref{e1-ex1}
satisfies the conclusions of the theorems in previous sections whose proofs rely on estimates
\eqref{e3} and \eqref{egrowth}.
Since the coefficients of \eqref{e1-ex1} is not Liptchiz, to obtain the desired results
we need to use \eqref{e3-ex1} and the following lemma.

\begin{lem}\label{lm6.2}
Let $R>0$.
For any $x,y\in\R^{r,\circ}_+$ and $|x|\vee|y|\leq R$.
\begin{equation}\label{e9-ex1}
\E\sup_{t\in[0,T]}\left\{|X^{x,\alpha}(t\wedge\tau)-X^{y,\alpha}(t\wedge\tau)|^2\right\}\leq K_{R,T}|x-y|^2.
\end{equation}
where  $K_{R,T}$
depends only on $R$ and $T$
and
$$\tau=\inf\{t\geq0: \alpha^{x,\alpha}(t)\ne\alpha^{y,\alpha}(t)\}.$$
\end{lem}
\begin{proof}
Using the elementary estimate
$|\diag(x)A(k)x-\diag(y)A(k)(y)|\leq C_A(|x|+|y|)|x-y|$
for some $C_A>0$,
we obtain
$$
\begin{aligned}
|X^{x,\alpha}&(t\wedge\tau)-X^{y,\alpha}(t\wedge\tau)|\\
\leq&
|x-y|+\sum_{i=1}^r\int_0^{t\wedge\tau} |b_i(X_i^{x,\alpha}(s)-X_i^{y,\alpha}(s))|dt
\\
&+C_A\int_0^{t\wedge\tau} \left(|(X^{x,\alpha}(s)|+|X^{y,\alpha}(s))|\right)|(X_i^{x,\alpha}(s)-X_i^{y,\alpha}(s))|ds
\\
&+\sum_{i=1}^r\left|\int_0^{t\wedge\tau} \sigma_i(\alpha(s))(X_i^{x,\alpha}(s)-X_i^{y,\alpha}(s))dW_i(s)\right|
\end{aligned}
$$
It follows from the Cauchy-Schwarz inequality that
\begin{equation}\label{e4-ex1}
\begin{aligned}
|X^{x,\alpha}&(t\wedge\tau)-X^{y,\alpha}(t\wedge\tau)|^2\\
\leq&
C|x-y|^2+C\left(\int_0^{t\wedge\tau} \left(1+|(X^{x,\alpha}(s)|+|X^{y,\alpha}(s))|\right)|(X^{x,\alpha}(s)-X^{y,\alpha}(s))|ds\right)^2
\\
&+C\sum_{i=1}^r\left|\int_0^{t\wedge\tau} \sigma_i(\alpha(s))(X_i^{x,\alpha}(s)-X_i^{y,\alpha}(s))dW_i(s)\right|^2
\end{aligned}
\end{equation}
for some constant $C>0$.
By the Burkholder-Davis-Gundy Inequality,
\begin{equation}\label{e5-ex1}\begin{aligned}
C\E\sup_{t\in[0,T]}\left\{\sum_{i=1}^r\left|\int_0^{t\wedge\tau} \sigma_i(\alpha(s))(X_i^{x,\alpha}(s)-X_i^{y,\alpha}(s))dW_i(s)\right|^2\right\}\\\ \leq \tilde C\E\int_0^{T\wedge\tau}|X^{x,\alpha}(s)-X^{y,\alpha}(s)|^2ds
\end{aligned}\end{equation}
for some constant $\tilde C$.
In view of H\"older's inequality,
\begin{equation}\label{e8-ex1}
\begin{aligned}
\E\sup_{t\in[0,T]}&\left\{C\left(\int_0^{t\wedge\tau} \left(1+|(X^{x,\alpha}(s)|+|X^{y,\alpha}(s))|\right)|(X^{x,\alpha}(s)-X^{y,\alpha}(s))|ds\right)^2\right\}\\
=&
 C\E\left(\int_0^{T\wedge\tau} \left(1+|(X^{x,\alpha}(s)|+|X^{y,\alpha}(s))|\right)|(X^{x,\alpha}(s)-X^{y,\alpha}(s))|ds\right)^2\\
\leq& C\left[\E\int_0^{T\wedge\tau}\left(1+|X^{x,\alpha}(t)|+|X^{y,\alpha}(t)|\right)^2dt\right]\E\int_0^{T\wedge\tau} |X^{x,\alpha}(t)-X^{y,\alpha}(t)|^2dt\\
\leq&K(1+T)(1+|x|^2+|y|^2)\E\int_0^{T\wedge\tau} |X^{x,\alpha}(t)-X^{y,\alpha}(t)|^2dt\,\, \text{ (due to \eqref{e3-ex1})}\\
\leq &K(1+T)(1+|x|^2+|y|^2)\E\int_0^{T} \sup_{s\in[0,t]}\left\{|X^{x,\alpha}(s\wedge\tau)-X^{y,\alpha}(s\wedge\tau)|^2\right\}dt
\end{aligned}
\end{equation}
for some $K>0$.
Taking the supreme over $[0,T]$, followed by taking the expectation on both sides of \eqref{e4-ex1}, and using \eqref{e5-ex1} and \eqref{e8-ex1}, we have
\begin{equation}\label{e6-ex1}
\begin{aligned}
\E&\sup_{t\in[0,T]}\left\{|X^{x,\alpha}(t\wedge\tau)-X^{y,\alpha}(t\wedge\tau)|^2\right\}\\
&\leq C|x-y|^2+K(1+T)(1+|x|^2+|y|^2)\E\int_0^{T} \sup_{s\in[0,t]}\left\{|X^{x,\alpha}(s\wedge\tau)-X^{y,\alpha}(s\wedge\tau)|^2\right\}dt
\end{aligned}
\end{equation}
for some constant $K>0$. Then \eqref{e9-ex1} follows from the Gronwall
inequality.
\end{proof}
Although the coefficients of \eqref{e1-ex1} are not globally Lipschitz,
the estimates \eqref{e3-ex1} and \eqref{e9-ex1} are sufficient for us  to
derive the following results.

\begin{thm}
Assume that $q_{ij}(\cdot)$, $i, j\in\M$ are bounded and continuous.
Let $(X^{x,\alpha}(t)),\alpha^{x,\alpha}(t))$ be the solution to \eqref{e1-ex1} and \eqref{tran} with initial value $(x,\alpha)\in\R^{r,\circ}_+\times\M$.
The following assertions hold:
\begin{enumerate}[{\rm 1.}]
\item $X^{x,\alpha}(t)$ is
twice continuously differentiable with respect to $x$ in probability.
If
in addition, $q_{kj}(x)$ satisfies \eqref{holder}
then $X^{x,\alpha}(t))$ is
twice continuously differentiable in $L^p$ with respect to $x$ for any $0<p<\lambda$, where $\lambda$ is the H\"older exponent in  \eqref{holder}.
\item If $q_{kj}(x)$ satisfies \eqref{qlip} then for any $R$ and $T>0$, there is a $C_{R, T}>0$ such that
 for any $x,\wdt x\in\R^{r,\circ}_+, |x|\vee|\wdt x|\leq R$, and $\al\in\M$, we have
\begin{equation*}
\E \sup\limits_{t\in[0,T]}|X^{\wdt x,\al}(t) - X^{x,\al}(t)| \le C_{R,T}|\wdt x - x|.
\end{equation*}
\item
Assume  that for each $i,j\in\M$, $q_{ij}(\cdot)\in C^2$ and $|D^\beta q_{ij}(\cdot)|$ are Lipschitz and bounded uniformly by some constant $M$ for $|\beta|\leq 2$.
Let $\phi(\cdot,i) \in C^2$ satisfy
\begin{equation*}
|D^\beta_x\phi(x,i)|\leq K(1+|x|^{\gamma}), i\in\M, |\beta|\leq 2.
\end{equation*}
Then, $u(t,x,i)=\E[\phi(X^{x,i}(t),\alpha^{x,i}(t))]$
is twice continuously differentiable with respect to the variable $x\in\R^{r,\circ}_+$.
\item The solution process $(X(t),\alpha(t))$
for the system given by \eqref{e1-ex1} and \eqref{tran}
is a Markov-Feller process.
\end{enumerate}
\end{thm}

\begin{proof}
In the proof of Theorem \ref{thm-1},
we use the global Lipschitz to obtain \eqref{e3}.
In this example,
the constant $K$, depending only on $T$, in \eqref{e3} is replaced by $K_{R,T}$,
which depends on both $R$ and $T$ (see \eqref{e9-ex1}).
Although \eqref{e9-ex1} is
slightly weaker than \eqref{e3},
it is still sufficient to
follow the proofs of Theorems \ref{thm-1} and
\ref{thm5.2}
to obtain the first and third claims of Theorem 6.3, respectively.
The second claim is derived from Theorem \ref{prop4.2}
with the
minor modification that the constant $C_T$ in \eqref{e5.1}
is replaced by $C_{R,T}$ because the constant $K_{R,T}$
in \eqref{e9-ex1} depends on $R$.
The forth claim follows directly from Theorem \ref{Feller}.
\end{proof}

\end{exm}

\end{document}